\documentclass[EJP,preprint]{ejpecp} 

\usepackage{caption}
\usepackage{stmaryrd}
 
\usepackage{apptools}

\usepackage{tikz}
\usepackage{floatrow}
\usepackage{xcolor}
\usetikzlibrary{arrows}
\usetikzlibrary{patterns}
\definecolor{ffqqqq}{rgb}{1,0,0}
\definecolor{qqffqq}{rgb}{0,1,0}
\definecolor{ffffff}{rgb}{1,1,1}
\definecolor{ttqqqq}{rgb}{0.07,0.07,0.07}
\colorlet{ColorGray}{gray!30}
\usetikzlibrary{shapes.misc}
\usetikzlibrary{shapes.geometric}
\tikzset{cross/.style={cross out, draw=black, minimum size=2*(#1-\pgflinewidth), inner sep=0pt, outer sep=0pt}, cross/.default={1pt}}

\usepackage{booktabs}

\usepackage{enumerate}  

\SHORTTITLE{Non-equilibrium East process}

\TITLE{On a front evolution problem
  for the multidimensional East~model}

\AUTHORS{%
 Yannick Couzini\'e\footnote{Dipartimento di Matematica e Fisica, Universit\`a Roma
  Tre, \EMAIL{yannick.couzinie@mailbox.org}}
  \and 
  Fabio Martinelli\footnote{Dipartimento di Matematica e Fisica, Universit\`a Roma
  Tre, \EMAIL{fabio.martinelli@uniroma3.it}}}


\AMSSUBJ{60K35; 82C20} 

\SUBMITTED{February 7, 2022} 
\ACCEPTED{October 19, 2022} 

\ARXIVID{2112.14693} 

\VOLUME{0}
\YEAR{2020}
\PAPERNUM{0}
\DOI{10.1214/YY-TN}

\ABSTRACT{We consider a natural front evolution problem for the East process on
$\bbZ^d, d\ge 2,$ a well studied
kinetically constrained model for which the facilitation mechanism is
oriented along the coordinate directions, as the equilibrium density $q$ of the
facilitating vertices vanishes.
Starting with a unique
unconstrained vertex at the origin, let $S(t)$  consist of those 
vertices which became unconstrained within time $t$ and, for an
arbitrary positive direction $\mathbf x,$ let $v_{\rm
  max}(\mathbf x),v_{\rm min}(\mathbf x )$ be 
the maximal/minimal velocities at which $S(t)$ 
grows in that direction. If $\mathbf x$ is independent of
$q,$ we prove that $v_{\rm max}(\mathbf x)= v_{\rm min}(\mathbf x)^{(1+o(1))}=\g_d ^{(1+o(1))}$ as $q\to 0,$ where $\g_d$ is the spectral gap of the
process on $\bbZ^d$.
We also analyse the case in which $\mathbf x$ depends on $q$ and some of its coordinates vanish as $q\to 0$.
In particular, for $d=2$ we prove that if $\mathbf x$ approaches one of the
two coordinate directions fast enough, then $v_{\rm max}(\mathbf x)= v_{\rm min}(\mathbf x)^{(1+o(1))}=\g_1 ^{(1+o(1))}=\g_d^{d(1+o(1))},$ i.e.\ the growth of $S(t)$ close to the coordinate
directions is much slower than the growth in the bulk and it is dictated by the
\emph{one dimensional} process. As a result the region $S(t)$ becomes
extremely elongated inside $\bbZ^d_+.$ We also establish mixing time cutoff for the chain in finite boxes with minimal boundary conditions. A key ingredient of our
analysis is the renormalisation technique of \cite{CFM2}
to estimate the spectral gap of the East process. A main novelty here is the extension of this
technique to get the main asymptotic as $q\to 0$  of a suitable principal Dirichlet eigenvalue of the process.}

\DeclareMathSymbol{\leqslant}{\mathalpha}{AMSa}{"36} 
\DeclareMathSymbol{\geqslant}{\mathalpha}{AMSa}{"3E} 
\DeclareMathSymbol{\eset}{\mathalpha}{AMSb}{"3F}     
\renewcommand{\b}{\beta}

\def\1{\ifmmode {1\hskip -3pt \rm{I}} \else {\hbox {$1\hskip -3pt \rm{I}$}}\fi}

\newcommand{\var}{\operatorname{Var}}

\newcommand{\tc}{\thinspace |\thinspace}
\newcommand{\id}{\mathds{1}}

\newcommand{\D}{\Delta}

\renewcommand{\b}{\beta}
\renewcommand{\l}{\lambda}
\renewcommand{\L}{\Lambda}

\renewcommand{\l}{\lambda}
\renewcommand{\a}{\alpha}
\renewcommand{\d}{\delta}
\renewcommand{\t}{\tau}

\newcommand{\g}{\gamma}
\newcommand{\G}{\Gamma}

\newcommand{\e}{\varepsilon}

\newtheorem*{theorem*}{Theorem}
\newtheorem{maintheorem}{Theorem}
\newtheorem{maincorollary}{Corollary}

\newtheorem*{question*}{Question}
\newtheorem*{remark*}{Remark}

\usepackage[capitalize]{cleveref}

\newcommand{\N}{\mathbb N}

\newcommand{\cA}{\ensuremath{\mathcal A}}
\newcommand{\cB}{\ensuremath{\mathcal B}}

\newcommand{\cD}{\ensuremath{\mathcal D}}

\newcommand{\cF}{\ensuremath{\mathcal F}}
\newcommand{\cG}{\ensuremath{\mathcal G}}
\newcommand{\cH}{\ensuremath{\mathcal H}}

\newcommand{\cL}{\ensuremath{\mathcal L}}

\newcommand{\cT}{\ensuremath{\mathcal T}}

\newcommand{\bbE}{{\ensuremath{\mathbb E}} }

\newcommand{\bbN}{{\ensuremath{\mathbb N}} }

\newcommand{\bbP}{{\ensuremath{\mathbb P}} }

\newcommand{\bbR}{{\ensuremath{\mathbb R}} }

\newcommand{\bbZ}{{\ensuremath{\mathbb Z}} }
\newcommand{\Z}{{\ensuremath{\mathbb Z}} }

\let\a=\alpha \let\b=\beta   \let\d=\delta  \let\e=\varepsilon
 \let\g=\gamma     \let\k=\kappa  \let\l=\lambda
      \let\o=\omega      
  \let\s=\sigma \let\t=\tau   
 \let\x=\xi 
\let\D=\Delta   \let\G=\Gamma  \let\L=\Lambda 
\let\O=\Omega      

\RequirePackage{lmodern}
\RequirePackage[T1]{fontenc}
\begin{document}






\section{Introduction}
The East\footnote{The nickname ``East'' here is only to keep up with
  the tradition. In two dimension ``South-or-West'' would be more appropriate. }
process on $\bbZ^d$ (see \cite{Aldous},\cite{East-Rassegna} and references
therein for $d=1,$ and \cite{CFM2,CFM3,Laure2} for $d\ge 2$), is a keynote example of the class of
\emph{facilitated interacting particle systems} or \emph{kinetically constrained
models} (KCM) which play an important role in several  qualitative
and quantitative approaches to describe the complex behaviour of
glassy dynamics (see e.g.\ \cite{CG} and references therein). It  
is the interacting particle system with state space $\O=\{0,1\}^{\bbZ^d}$
(a continuous time Markov chain on $\{0,1\}^\L$ if restricted to a
finite $\L\subset \bbZ^d$) which
is informally described as follows. Each vertex
$x\in \bbZ^d$, with rate one and independently across $\bbZ^d$, is
resampled from $\{0,1\}$ according to the Bernoulli($p$)-measure,
$p=1-q,$ iff the current state carries at least one vacancy (i.e.\ a
state $``0"$) among the
neighbours of $x$ of the form $y=x-\mathbf e,\, \mathbf e\in \cB, $
where $\cB=(\mathbf e^{(1)},\dots,\mathbf e^{(d)})$ is the canonical basis
of $\bbZ^d$. The product Bernoulli($p$) measure on $\O$ is a
reversible measure for this process and the parameter $q$ is the equilibrium density of
the vacancies, i.e.\ of the facilitating vertices. In the physical
applications $q\simeq e^{-\b},$ where $\b$ is the inverse temperature.  

Thanks to the oriented character of its kinetic constraint (i.e.\ the
requirement that has to be fulfilled in order to permit the update of a vertex), the East
process is one of the few KCM for which a rigorous analysis of the actual
evolution of the process with some arbitrary initial distribution has
been accessible for any value of $q\in (0,1)$
\cite{Blondel,CMST,CFM,CFM2,CFM3,FMRT-cmp,Laure1,Laure2}. In this
paper, building in particular on \cite{CFM2, CFM3}, we make some progress in the analysis of a natural
front evolution problem in $\bbZ^d_+=\{x=(x_1,\dots,x_d)\in \bbZ^d:\
x_i\ge 0\}$ for $q\ll 1$ (i.e.\ low temperature) and $d\ge 2$. We refer the reader
to Section 2 for a precise formulation of the problem and of the main results.

\subsection{Notation}\ 
\begin{enumerate}[$\bullet$]
\item Let $\bbR^d_+=\{x=(x_1,\dots,x_d)\in \bbR^d:\ x_i\ge 0\}$ and for any
$x\in \bbR^d_+$  let $\lfloor
x\rfloor \in \bbZ^d_+$ be such that $\lfloor
x\rfloor_i= \lfloor
x_i\rfloor\ \forall i$. Unit vectors of $\bbR^d_+$ will be written in
bold. Given $x,y\in \bbZ^d_+$ we will write $x\prec
y$ iff $x_i\le y_i\ \forall i,$ $x\prec V, V\subset \bbZ^d_+,$ if
$x\prec y \ \forall y\in V,$ and $\|x-y\|_1:=\sum_i|x_i-y_i|$ for
their $\ell_1$-distance. We shall also write $x=0$ to denote the origin of
$\bbZ^d_+$.  
\item For any $\L \subset
\bbZ_+^d$ we define its \emph{oriented boundary}   $\partial_{\downarrow} \L $ as 
$
\partial_{\downarrow} \L\triangleq \{ x \in \bbZ_+^d \setminus \L\,:\,
x+\mathbf e \in \L \text{ for
  some } \mathbf e\in \cB\}.$ Notice that vertices of $\bbZ^d\setminus
\bbZ^d_+$  are \emph{not} part of the oriented boundary.
\item $\O_\L$ will denote for the product space
$ \{0,1\}^\L$ endowed with the product topology.  If $\L = \bbZ_+^d$ we
simply write $\O$. We will write
$\o_x\in \{0,1\}$ for the state at $x\in \L$ of the configuration
$\o\in \O_\L$ and we will refer to the vertices of $\L$
where $\o\in \O_\L$ is equal to one (zero) as
the \emph{particles} (\emph{vacancies}) of $\o$.  If $V\subset \L$
we will write $\o\restriction_{V}$ for the
restriction of $\o\in \O_\L$ to $V$. In particular we will write
$\o\restriction_V=1$ if $\o(x)=1\ \forall\ x\in V$.  
\item For any $\L\subset \bbZ_+^d$, a configuration
$\s\in \O_{\partial_{\downarrow} \L}$ will be referred to as a \emph{boundary
  condition for $\L$.} If $\s$ contains no particles it will be referred to as
\emph{maximal boundary condition}. Finally, for any given boundary
condition $\s\in \O_{\partial_{\downarrow} \L}$ and $\o\in \O_\L,$ we
will write $\s\cdot\o \in \Omega_{  \partial_{\downarrow} \L\cup \L}$ for the configuration equal
to $\s$ on $\partial_{\downarrow} \L$ and to $\o$ on $\L$.
\item Given $\L\subset \bbZ^d_+$ we will write $\mu_\L$ for the product 
Bernoulli($p$) measure on $\O_\L$ and
$\mu_\L(f),\var_\L(f)$ for the average and variance of $f:\O_\L\mapsto
\bbR$ w.r.t.\ $\mu_\L$. As for $\Omega_{\L}$, if $\L=\bbZ^d_+$ we omit the subscript $\L$
from the notation.
\end{enumerate}

\subsection{The \texorpdfstring{$d$}{d}-dimensional East process}
\label{sec:East process}
 Given $\L \subset \bbZ_+^d,\ \s\in \O_{\partial_{\downarrow} \L}$ and $\o\in \O_\L$, define the
\emph{constraint} $c_x^{\L,\s} (\o)$ at $x\in \L$ as
\[
  c_x^{\L,\s} (\o)=
  \begin{cases}
    1 & \text{if either $x=0$ or $\exists\ \mathbf e\in \cB:
      x-\mathbf e\in \partial_{\downarrow} \L\cup \L \text{ and } (\s\cdot\o)(x-\mathbf e)=0$},\\
    0 & \text{otherwise}.
  \end{cases}
\]
\begin{remark}
Notice that the origin is \emph{unconstrained}.  
\end{remark}

The
infinitesimal generator $\cL^\s _\L$ of the East process in $\L$ with vacancy density parameter $q \in (0,1)$ and boundary
 configuration $\s$ has the form
 \begin{align}
  \label{eq:gen}
  \cL^\s _\L f(\o) &= \sum _{x \in \L} c_x^{\L, \s}(\o)\, \bigl[
  \o_x q+ (1-\o_x)p \bigr] \cdot \bigl[ f(\o^x) -f(\o) \bigr]
  \nonumber\\
  &= \sum _{x \in \L} c_x^{\L, \s}(\o) \bigl[\mu_x (f)- f\bigr]
  (\omega),
\end{align}
where $\o^x$ is the configuration in $\O_\L$ obtained
from $\o$ by flipping its value at $x$. We refer the
reader to \cite{CMRT}.
As the local constraint $c_x^{\L, \s}(\cdot)$ does not depend on
the state of the process \emph{at $x$}, $\mu_\L$ is a reversible measure. Actually, thanks to the orientation
of the constraints a stronger property of local stationarity holds
\cite[Proposition 3.1]{CFM3} together with local exponential ergodicity (see \cite[Theorem 4.1]{CFM3} and \cite[Theorem 2.2]{Laure2}).
When the initial law of the process is
$\nu$ we will write $\bbP^{\L,\s}_\nu(\cdot), \bbE^{\L,\s}_\nu(\cdot)$ for the law and the associated expectation of
the process. When $\nu$ is the Dirac mass at one configuration $\o$ we
will simply write $\bbP^{\L,\s}_\o(\cdot)$ and
$\bbE^{\L,\s}_\o(\cdot)$. The superscript $\L$ will be dropped from the notation if
$\L=\bbZ^d_+$. Similarly for the superscript $\s$ if
$\partial_{\downarrow}\L=\emptyset$. Finally, $\cD_\L^\s(f),\ f:\O_\L\mapsto
\bbR$ denotes the Dirichlet form of the process (i.e.\ the quadratic form of $-\cL_\L^\s$). By construction,
$\cD_\L^\s(f)=\sum_{x\in \L}\mu_\L\big(c_x^{\L,\s}\var_x(f)\big)$.
\begin{remark}
  \label{rem:1}
For $d\ge 2$ and any integer $d'\in [1,d-1]$ the
projection of the East process on $\bbZ_+^d$ onto 
$\bbZ_+^{d'}=\{x\in \bbZ_+^d:
 x_j=0 \ \forall j>d'\}$ coincides with the East process on
 $\bbZ_+^{d'}$. Similarly, for any finite $V\subset \bbZ^d_+$ and any
 box $\L\supset V$ the
 projection of the East process
 on $\bbZ^d_+$ onto $V$ coincides with the same projection of the
 East chain on $\L$.
 \end{remark}
 \subsection{Structure of the paper}
 \begin{itemize}
     \item In Section 2 we formulate the front evolution problem on the positive quadrant of $\bbZ^d$ and state our main result as $q\to 0$ on smallest/largest front velocity in a given direction (cf. Theorem \ref{thm:1}). In turn, Theorem \ref{thm:1} implies the main result on the local equilibrium behind the front (cf. Theorem \ref{thm:2}) together with the mixing time cutoff for the East chain on a box with sides along the coordinate axes (cf. Theorem \ref{thm:3}). 
     \item In Section 3 we develop the two main technical tools needed for the proof of the main results, namely a sharp lower bound on a suitable Dirichlet eigenvalue of the Markov generator (cf. section \ref{sec:RG eigenvalue}) and a bottleneck result (cf. Section \ref{sec:bottleneck}).
     \item Section \ref{sec:proofs} is devoted to the proof of the three main theorems, while Section \ref{sec:Proof of Proposition 3.1} contains the proof of Proposition \ref{prop:1}, the key technical result from Section 3.  
     \item Finally the Appendix contains the proof of a couple lemmas.
 \end{itemize}
\section{The front evolution problem and main result}
Let $\o^*\in \O$ be the configuration identically to one and 
write $\t_x, x\in \bbR^d_+,$ for the hitting time of the set $\{\o:\ \o_{\lfloor
x\rfloor}=0\}$. Sometimes we will refer to $\t_x$ as the \emph{infection time} of $x$. More generally, for any $A\subset \bbZ^d_+$ we will write $\t_A$ for the
hitting time of the set $\{\o:\ \o\restriction_A\neq 1\}.$ Given a unit vector $\mathbf x\in \bbR^d_+,$ it is known 
\cite[Theorem 5.1]{CFM3} that for any $q\in (0,1)$ 
\begin{equation}
  \label{eq:1}
  \bbE_{\o^*}(\t_{n\mathbf x})= \Theta(n), \quad \text{as $n\to +\infty$},
\end{equation}
and that the mixing time of the East chain in
$\{0,\dots,n-1\}^d$ is $\Theta(n)$. It is then natural to define 
\[
\frac{1}{v_{\rm max}(\mathbf x)}=\liminf_{n\to 
\infty}\frac{\bbE_{\o^*}\big(\t_{n\mathbf x}\big)}{n}, \qquad \frac{1}{v_{\rm min}(\mathbf x)}=\limsup_{n\to 
\infty}\frac{\bbE_{\o^*}\big(\t_{n\mathbf x}\big)}{n},
\]
and
denote them  as the maximal and minimal front velocity in the
direction of $x$ respectively. Using \eqref{eq:1} $0<v_{\rm min}(\mathbf x)\le
v_{\rm max}(\mathbf x)<+\infty$ for all $\mathbf{x}$.
\begin{remark}
\label{rem:0}Using the strong Markov property and subadditivity, it is not difficult to see that
$\hat v(\mathbf x)^{-1}:=\lim_{n\to
    \infty}\max_{\o}\bbE_\o(\t_{n\mathbf x})/n$
  exists. Clearly $v_{\rm min}(\mathbf x)\ge \hat v(\mathbf x).$
\end{remark}
In analogy with the classic
\emph{shape theorem} for e.g.\ first passage percolation (see e.g.\ \cite{ADH})
we conjecture that $v_{\rm max}(\mathbf x)=v_{\rm min}(\mathbf
x):=v(\mathbf{x})$
and in that case $v(\mathbf{x})$ represents the
\emph{front velocity in the direction $\mathbf{x}$}. Similarly, for any $t>0$ we could define
the random set (see \cref{fig:0})
\[
  S(t)= \{x\in \bbR^d_+:\ \t_x\le t\},
  \]
and conjecture that there exists a compact subset $\hat S\subset
\bbR^d_+$ such that
\[
\forall \, \e>0\quad   \lim_{t\to \infty} \bbP_{\o^*}\big((1-\epsilon)t\hat S\subseteq
  S(t) \subseteq (1+\epsilon)t\hat S \big)=1.
\]
\begin{figure}
\centering
\begin{tikzpicture}
\node at (7.8,7.8)
    {\includegraphics[scale=0.35]{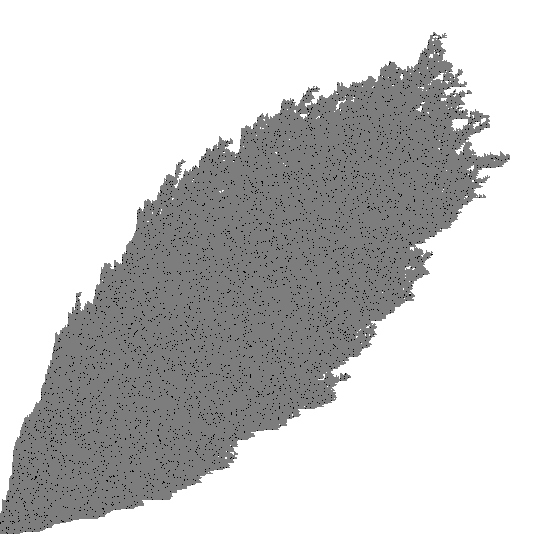}};
    \draw (4.5,4.5) rectangle (11,11);
  \end{tikzpicture}
\caption{A simulation of the random set $S(t)$ for $q=0.04$ suggesting
  the existence of a limit shape. The grey region corresponds to vertices that have been updated at least once before time $t$, while the black dots denote the actual infected sites at time $t$.}
\label{fig:0}
\end{figure}
\begin{remark}
\label{rem:one dim}Using coupling arguments, it has been proved for $d=1$
\cite{Blondel} that $\forall q\in (0,1)$ the position $\xi_t$ of the
rightmost vacancy for
the process started from $\o^*$ obeys a law
of large numbers $\lim_{t\to \infty} \xi_t/t= v \ \text{a.s.}$ and that the law of the East process to the left of $\xi_t$
converges exponentially fast to a limiting law. A precise CLT for
$\xi_t$ was later proved in
\cite{Ganguly-Lubetzky-Martinelli} together with a cutoff result for
the mixing time in a finite interval. In particular, for $d=1$ both conjectures are known to be true. 
For $d\ge 2,$
Remark \ref{rem:1} together with the law of large numbers in $d=1$ imply that
$v_{\rm max}(\mathbf e)=v_{\rm min}(\mathbf e)=v \ \forall \mathbf e \in \cB.$ For all other
directions both conjectures are still widely open. 
\end{remark}
In this paper, for any $d\ge 2$ we
provide a contribution towards the understanding of the front
evolution problem as the vacancies equilibrium density $q\to 0.$ Specifically, our main result concerns the small $q$
behaviour of $v_{\rm max}(\mathbf x),v_{\rm min}(\mathbf x)$ as a function of
$\mathbf{x}\in \bbR^d_+$. We will distinguish between the case
in which the direction $\mathbf{x}$ is fixed independent of $q$ and all its coordinates are positive, and the case in
which $\mathbf{x}=\mathbf{x}(q)$ and $\min_i \mathbf{x}_i\to 0$ 
as $q\to 0$. In the first case we will say that $\mathbf x$ points towards the \emph{bulk of $\bbR^d_+$,} while in the second case $\mathbf x$ points to the \emph{boundary of $\bbR^d_+$}. In the sequel $\theta_q:=|\log_2 q|$ will be the relevant
parameter.
  \begin{maintheorem}
    \label[theorem]{thm:1} Fix $d\ge 2$.  
  \begin{enumerate}[(A)]
  \item Let $\mathbf x\in \bbR^d_+$ be a unit vector with $\min_i\mathbf{x}_i>0$. Then
\[
  \lim_{q\to 0}-\frac{2}{\theta_q^2}\log_2(v_{\rm max}(\mathbf x))=\lim_{q\to 0}-\frac{2}{\theta_q^2}\log_2(v_{\rm min}(\mathbf x))=\frac{1}{d}.
\]
\item Let $0<\b< 1, \k\ge 1 $ and let $\{\mathbf x(q)\}_{q\in (0,1)}$ be a
  family of unit
  vectors in $\bbR^d_+$ such that $\max_{i,j} \mathbf x_i(q)/\mathbf x_j(q)\le
  \k 2^{\b\theta_q}.$  Then
\[
 \limsup_{q\to 0}-\frac{2}{\theta_q^2}\log_2(v_{\rm min}(\mathbf x(q)))<1.
\]
\item Assume $d=2$ and let $\a>0.$ Let $\{\mathbf x(q)\}_{q\in (0,1)}$ be a
  family of unit
  vectors in $\bbR^2_+$ such that $\max_{i,j} \mathbf x_i(q)/\mathbf x_j(q) \ge
  2^{\a \theta_q^2}$. 
 Then
    \begin{align*}
\liminf_{q\to 0}-\frac{2}{\theta_q^2}\log_2(v_{\rm max}(\mathbf x(q))) \ge \frac{(1+4\a) \wedge 2}{2}.
    \end{align*}
Moreover, if $\a>1/4$ then
   \[
  \lim_{q\to 0}-\frac{2}{\theta_q^2}\log_2(v_{\rm max}(\mathbf x(q)))=\lim_{q\to 0}-\frac{2}{\theta_q^2}\log_2(v_{\rm min}(\mathbf x(q)))=1.
\]
  \end{enumerate}
  The same results apply to $\hat v(\mathbf x)$ defined in Remark \ref{rem:0}.
\end{maintheorem}
\begin{remark}
  Part (C) is presented here only for $d=2$ for simplicity. Remark \ref{rem:1} and
the same proof ideas give similar, although more involved, results also for
$d\ge 3$.   
\end{remark}
By combining (A) above together with Remark \ref{rem:1} we immediately
get
\begin{maincorollary}
Fix $d\ge 2$ and  let $\mathbf x\in \bbR^d_+$ be a unit vector such
that $\min_i \mathbf x_i=0$. Then
\[
   \lim_{q\to 0}-\frac{2}{\theta_q^2}\log_2(v_{\rm max}(\mathbf x))=\lim_{q\to 0}-\frac{2}{\theta_q^2}\log_2(v_{\rm min}(\mathbf x))=\frac{1}{d(\mathbf x)},
 \]
 where $d(\mathbf{x}):=\#\{i\in [d]:\ \mathbf{x}_i >0\}$\footnote{Here
  $[n]:=\{1,2,\dots,n\}$ for any positive integer $n$}.
\end{maincorollary}
\begin{remark}
In order to better understand \cref{thm:1}, let us recall a key
feature of the East process on the \emph{full lattice} $\bbZ^d, d\ge 1.$ It is a reversible process
with a positive spectral gap $\g_d$ satisfying (see \cite{Aldous,CMRT} for $d=1$ and \cite{CFM2} for
$d\ge 2$): 
\[
  \lim_{q\to 0} -\frac{2}{\theta_q^2}\log_2(\g_d)=1/d.
\]
Notice that $\g_{d+1}=\g_d^{\,(1+o(1))d/(d+1)}$. Then the three statements of the theorem can be interpreted respectively as follows:
\begin{enumerate}[(A)]
\item if the direction $\mathbf{x}$ points towards the
\emph{bulk} of $\bbR^d_+$, then
$v_{\rm max}(\mathbf x)=v_{\rm min}(\mathbf x)^{1+o(1)}=\g_d ^{1+o(1)}$;
\item if $\mathbf x=\mathbf x(q)$ points to the boundary of $\bbR^d_+$  \emph{slowly enough} as $q\to 0$, then $v_{\rm min}(\mathbf x)$ is much larger than the velocity $
v(\mathbf e), \mathbf e \in \cB,$ in any coordinate direction;
\item for $d=2$ if $\mathbf x=\mathbf x(q)$ points to the boundary of $\bbR^2_+$ \emph{fast
  enough}, then $v_{\rm max}(\mathbf x)$ is much smaller than
the minimal velocity associated to any direction pointing to the bulk of $\bbR^2_+$ and,  if sufficiently fast then $v_{\rm max}(\mathbf x)=
v_{\rm min}(\mathbf x) ^{1+o(1)}=v(\mathbf e) ^{1+o(1)}, \mathbf e\in \cB.$
\end{enumerate}
\end{remark}
\begin{remark}
\cref{thm:1} has been largely motivated by \cite[Theorem 3]{CFM2}. There the authors considered $\L=\{0,\dots,L\}^d, \bbN\ni L\le 2^{\theta_q/d},$ and, using capacity methods combined with a sophisticated combinatorial analysis, analysed the asymptotic behaviour as $q\to 0$ of the mean hitting time $\bbE_{\o^*}(\t_x)$ for two special vertices:
$x_\L=(L,\dots,L)$
and $x'_\L=(L,0,\dots,0).$ One of the main outcomes was that for
$L=2^{\theta_q/d}$ and as $q\to 0$ 
$\bbE_{\o^*}(\t_{x'_\L})=\bbE_{\o^*}(\t_{x_\L})^{d(1+o(1))}$. In other words, for $q$ small enough
and at the length scale $2^{\theta_q/d},$ there is a big time scale
separation between the two mean hitting times. The restriction $L\le 2^{\theta_q/d}$ was dictated by the need of having at equilibrium a constant number of vacancies in the box $\L$ and it was basically unavoidable.  

Extending the analysis of the mean hitting time $\bbE_{\o^*}(\t_x)$ to vertices $x$ of the form $x = n\mathbf x$, where $\mathbf x $ is any direction of $\bbR^d_+$ and $n\in \bbN$ is \emph{arbitrary}, using capacity methods as in \cite{CFM2} seems prohibitive. Therefore, in order to prove \cref{thm:1} we must to appeal to large deviations combined with a fine analysis of certain principal Dirichlet eigenvalues of the process using the renormalization group ideas developed in \cite{CFM2}. The latter technique is illustrated in \cref{sec:RG eigenvalue}.
\end{remark}
The second result analyses the law at time $t\gg 0$ of the East process with initial condition
$\o^*$. It proves that for $q$ small enough the region of $\bbZ^d_+$ where the
East process at time $t$ has relaxed to the reversible measure $\mu$
is extremely elongated in the bulk of $\bbZ^d_+$ (see \cref{fig:0}). 
\begin{maintheorem}
  \label[theorem]{thm:2}Fix $d\ge 2, 0\le \d <1$ and $\e>0$. Let
  \[
    \L(\d,\e,t)=\{x\in \bbZ^d_+:
\min_{i,j}x_i/x_j \ge \d \text{ and } \|x\|_1\le 
2^{-\frac{\theta_q^2}{2d}(1+\e)}\times t\},\quad t>0,
\]
and let
$\nu^{\d,\e}_t$ be
the marginal on $\O_{\L(\d,\e,t)}$ of the law of the East process at time $t$
with initial condition $\o^*$. Then, 
\begin{align}
  \label{eq:2}
\limsup_{\e\to 0}\limsup_{q\to 0}\limsup_{t\to \infty}\|\nu_t^{\d,\e} -
  \mu_{\L(\d,\e,t)}\|_{TV}&=0 \quad \text{if $\d>0,$}\\
\label{eq:2.1}  \liminf_{\e\to 0}\liminf_{q\to 0}\liminf_{t\to \infty}\|\nu_t^{\d,\e} -
  \mu_{\L(\d,\e,t)}\|_{TV}&=1 \quad \text{if $\d=0.$}
  \end{align}
\end{maintheorem} 
\begin{remark}
A slightly more refined formulation of \cref{thm:2} avoiding
the $\limsup$ on $\e,q$
would have been possible. However, we
opted for the present version for simplicity.      
\end{remark}
Finally we analyse the \emph{mixing time} (see e.g. \cite{Levin2008})
of the East chain on the sequence of boxes $\L_n=\{0,\dots,n\}^d, d\ge 2$. For $q$ small enough and any $n$ large enough we prove total variation cutoff -- i.e. a sharp transition in mixing (see \cite{AldousDiaconis, Diaconis} and references therein) -- around the time 
\begin{equation}
\label{eq:mixing time}
    T_n= n/v,
\end{equation} where $v$ is the front velocity along any coordinate direction $\mathbf e\in \cB$ (see Remark \ref{rem:one dim}).  
More precisely, let 
$d_n(t)=\max_{\o\in \O_{\L_n}}\|\bbP_\o^t(\cdot)-\mu_{\L_n}\|_{TV}$,
where $\bbP_\o^t(\cdot)$ denotes the law at time $t$ of the East process on $\L_n$ with initial condition $\o$.
\begin{maintheorem}
\label[theorem]{thm:3}There exists $q_0\in (0,1)$ such that for any $0<q\le q_0$ 
\begin{align}
\label{eq:cutoff1}    \lim_{\a\to \infty}\liminf_{n\to +\infty}d_n(T_n-\a \sqrt{n})&=1\\
\label{eq:cutoff2}  \limsup_{n\to +\infty}d_n(T_n+n^{2/3} )&=0  
\end{align}
\end{maintheorem}
\begin{remark}
Above we didn't try to optimise the cutoff window size. Using \cite[Theorem 2]{Ganguly-Lubetzky-Martinelli} $T_n$ is the mixing time of the standard one dimensional East chain on the interval $\{0,\dots,n\}$. Hence, in a very precise sense, the one dimensional evolution along the coordinate axes dominates the mixing process of the multidimensional East chain in $\L_n$.  
\end{remark}
\cref{thm:3} may look a bit surprising given that we don't know the existence of the front velocity in any direction $\mathbf x$. However, here we exploit the geometry of the boxes $\L_n$ together with the chosen boundary conditions for the East chain (only the origin is unconstrained), and the fact that for small $q$ the front velocity along the coordinate axes is much smaller than the minimal velocity in any other direction pointing towards the bulk of $\L_n$ (cf. part A of \cref{thm:1}). A cutoff result with e.g. a different choice of the geometry of $\L_n$ or of the boundary conditions (e.g. any vertex on the coordinate axes is unconstrained) would require proving at least the existence of the front velocity.

\section{Two key tools}
In this section we describe the two main tools that we use in order to
get upper and lower bounds on $v_{\rm max}(\mathbf x), v_{\rm min}(\mathbf x)$. 
\subsection{Lower bounds on a Dirichlet eigenvalue}
\label{sec:RG eigenvalue}In the sequel we adopt the following convention for the process on
$\L\subset \bbZ^d_+$ with boundary condition $\s$. If either $\s$
is absent because $\partial_{\downarrow}\L=\emptyset$ or
$\s\equiv 1$, then the superscript $\s$ is dropped from the notation.
Given integers
$(L_1,\dots,L_d)$ the set $\L=\prod_{i=1}^d \{0,\dots, L_i\}$ will be called the \emph{box} with
side lengths $(L_1,\dots,L_d).$ We will write $x_\L$ for the vertex $(L_1,\dots,L_d).$ Notice that
$\partial_{\downarrow}\L=\emptyset$. Given a box $\L$ with
side lengths $(L_1,\dots,L_d)$ the set $x + \L$ will be called the \emph{box} with
side lengths $L_1,\dots,L_d$ and origin at $x$. Unless otherwise
specified a box will always have its origin at $x=0$.

Recall now that the origin is always unconstrained. Given a box $\L$ possibly depending on $q,$ it is well
known (see e.g.\ \cite[Section 6]{AB}) that the
hitting time $\t_{x_\L}$ satisfies
\begin{equation}
  \label{eq:4}
  \bbP_\mu(\t_{x_\L}>t)\le e^{-\l^D(\L)t},
\end{equation}
where 
\begin{equation}
  \label{eq:27}
\l^D(\L)=\inf\{\cD_\L(f)/\mu_\L(f^2): \ f:\O_\L\mapsto \bbR,\ f\restriction_{\{\o:\o_{x_\L}=0\}}=0\}
\end{equation}
is the smallest eigenvalue for the Dirichlet problem
\[
  -\cL_\L f = \l f, \quad f\restriction_{\{\o:\ \o_{x_\L}=0\}}=0.
\]
A lower bound on $\l^D(\L)$ is obtained via the
spectral gap $\g(\L)>0$ of the East
chain in $\L$. Using $\var_{\L}(f) \ge q\mu_\L(f^2)$ for all $f$ such that
$f\restriction_{\{\o:\o_{x_\L}=0\}}=0$, we get immediately 
\begin{equation}
  \label{eq:2bis}
\l^D(\L)\ge q\, \g(\L).
\end{equation}
Using Lemma \ref{lem:East}
it follows that $\g(\L)=\g_{d=1}^{(1+o(1))} $
as soon as $\max_i L_i\ge 2^{\theta_q}$ because of the slow relaxation
process mode along the
edges of $\L$ on the coordinate axes. 

If $\max_{i,j}(L_i\vee 1)/(L_j\vee 1) =O(1)$ as $q\to
  0,$ \eqref{eq:2bis} is a very pessimistic bound when $d\ge 2$ because $\l^D(\L)$
should be mostly influenced by the \emph{$d$-dimensional bulk dynamics} rather than
by the \emph{one dimensional dynamics} along the edges of $\L$. In this case 
it is natural to conjecture that, to the leading order as $q\to 0,$
$\l^D(\L)$ is lower bounded by $\g_d$. In order to prove the conjecture the following provides a
better bound than \eqref{eq:2bis}.

For any $V\subset \bbZ^d_+$ let $\g(V)$ be the spectral gap of the East chain in $V$ with boundary conditions identically
        equal to $1$ on $\partial_{\downarrow}V$.  
        \begin{claim}
          \label{claim:0}
          \begin{align}
            \label{eq:00}
  \l^D(\L)&\ge \max\{\l^D(V):\, V\subseteq \L,\,  V\supset \{0,
            x_\L\}\}\nonumber\\
            & \ge q\max\{\g(V):\, V\subseteq \L,\,  V\supset \{0, x_\L\}\}>0.
            \end{align}
\end{claim}
    \begin{proof}[Proof of the claim]
Clearly $\max\{\g(V):\, V\subseteq \L,\,  V\supset \{0, x_\L\}\}\ge \g(\L)>0$.
Now fix $\L\supseteq V\ni \{0, x_\L\}$ together with $f$ such that $f\restriction_{\{\o:\o_{x_\L}=0\}}=0,$ and observe that monotonicity in
the constraints implies that
        \[
          \cD_\L(f)\ge \sum_{\o\in \O_{\L\setminus
              V}}\mu_{\L\setminus V}(\o)
            \cD_V(f(\o\,\cdot)).
          \]
Since $V\ni x_\L,$ for any $\o \in \O_{\L\setminus V}$ the function $\O_V\ni \o'\mapsto f(\o\cdot \o') $ vanishes if $\o'_{x_\L}=0$. Therefore, \eqref{eq:27}
  implies that for any $\o\in \O_{\L\setminus V}$
  \[\cD_V(f(\o\,\cdot))\ge \l^D(V)\mu_V(f^2(\o\,\cdot).
  \]
By averaging over $\o$ both sides of the above inequality w.r.t.\ $\mu_{\L\setminus V}(\o)$ we conclude that $\cD_\L(f)\ge \l^D(V)\mu_\L(f^2)$ and the first inequality of
 the claim follows.
 The second inequality follows from the general inequality \eqref{eq:2bis}. 
  \end{proof}
In order to bound from below the r.h.s. of \eqref{eq:00} according to
whether $\max_{i,j}(L_i\vee 1)/(L_j\vee 1) =O(1)$ as $q\to
  0$ or not, it is
convenient to introduce the following geometrical definition.
  \begin{definition}
Fix $d\ge 2, \b \ge 0,$ and $ \kappa\ge 1$. For any given $q\in (0,1)$ let
$S(\b,\kappa;\theta_q)$ be the
collection of $d$-tuple of integers $(L_1,\dots,L_d)$ such that
$\max_{i,j}(L_i\vee 1)/(L_j\vee 1)\le \kappa 2^{\b \theta_q}$.  We say that a box $\L$ with side lengths
$(L_1,\dots,L_d)$ is $(\b,\kappa;\theta_q)$-outstretched if
$(L_1,\dots,L_d)\in S(\b,\kappa;\theta_q)$, i.e. the maximum aspect ratio between its sides does not exceed $\kappa 2^{\b \theta_q}$.
 Notice that $S(\b,\kappa;\theta_q)\subseteq S(\b',\kappa;\theta_q)$ if $\b\le \b'$.
 \end{definition}
 \begin{remark}
Although the class of $(\b,\kappa;\theta_q)$-outstretched boxes contains very regular boxes, e.g. cubes, our focus will be on the most extreme cases where the aspect ratio between the box's sides is close to  $\kappa 2^{\b \theta_q}$.     
 \end{remark}
In the sequel, the parameters $\b,\k$ will always be chosen
independent of $q$. Moreover, whenever the value of $q$ is understood we will simply
write $(\b,\kappa)$-outstretched instead of $(\b,\k;\theta_q)$-outstretched. 
\begin{definition}
\label{def:2}    
Given $\b \ge 0$ we say that $\l>0$ satisfies condition
$\cH(\b)$ and write $\l\sim \cH(\b)$ if for any $\kappa\ge 1, \e>0$ there exists
$q(\b,\kappa,\e)>0$ such that $\forall q\le q(\b,\kappa,\e)$  the
following occurs: $\forall\ (\b,\kappa;\theta_q)$-outstretched box $\L$ 
$\exists\ V\subset \L$ with $V\supset \{0,x_\L\}$
such that $\g(V)\ge  2^{- (1  +\e)\l\frac{\theta_q^2}{2}}.$
We then let $\phi(\b;d)=\min\{\l>0: \l \sim 
\cH(\b)\}
$.
\end{definition}
\begin{remark}
For $d=1$ any box $\L_L=\{0,1,\dots,L\},$ is $(\b,\k)$-outstretched for all
$\b\ge 0,\kappa \ge 1$. Therefore, $\phi(\b;1)=1$ because
$\inf_L\g(\L_L)=2^{-\frac{\theta_q^2}{2}(1+o(1))}$ \cite{CMRT}.   
\end{remark}
Thus, if $\l\sim \cH(\b)$ then Claim \ref{claim:0} implies that for all
$\e>0$ the Dirichlet eigenvalue $\l^D(\L)$ is greater than $2^{- (1
  +\e)\l\frac{\theta_q^2}{2}}$ for \emph{all} $(\b,\kappa;\theta_q)$-outstretched
box $\L$ and for all $q$ small enough depending \emph{only} on $\b,\k,\e$. In
particular,
\begin{equation}
  \label{eq:4bis}
\l^D(\L)\ge 2^{- (1 +\e) \phi(\b;d)\frac{\theta_q^2}{2}}.
\end{equation}
A major problem is then to bound the constant $\phi(\b;d)$ for $d\ge 2$. Lemma \ref{lem:East} implies that $\phi(\b,d)\le
1$. The next result, which in a sense represents the technical core of the paper
and whose proof is
deferred to Section \ref{sec:Proof of Proposition 3.1}, goes beyond
this bound. 
\begin{proposition}\ 
  \label{prop:1} For $d\ge 2$ the coefficient $\phi(\b;d)$ satisfies:
  \begin{align*}
    &(i) &\phi(0;d)
      &=1/d;\\
     &(ii) & \phi(\b;d) &<1 \quad \forall \b\in (0,1);\\
     & (iii) &\phi(\b;d) &=1 \quad  \forall \b\ge 1.
    \end{align*}
In particular, for any $d\ge 2$ and any $(\b,\k)$-outstretched box $\L$
with $\b<1$ the Dirichlet eigenvalue $\l^D(\L)\gg \g_{d=1}$ as $q\to 0$.      
  \end{proposition}
A first consequence for the hitting times $\t_x,\ x\in \bbZ^d_+,$ is
provided by the next result. 
\begin{lemma}
  \label{lem:1}
Fix $\e>0, \b\ge 0,\k\ge 1.$ Then there exists $q(\e,\b,\k)$ such that for any $q\le q(\e,\b,\k)$ and any $\L=\L_q$ a
$(\b,\k;\theta_q)$-outstretched box of side lengths $(L_1,\dots,L_d)$ satisfying $2^{\theta_q^{3/2}}/2 \le \min_i
L_i \le 2^{\theta_q^{3/2}}$, the following holds: 
\[
\sup_{x\in \bbZ^d_+} \sup_{\o\in \{\o:\,\o_x=0\}}\bbE_\o(\t_{x+x_\L}) \le 2^{(1+\e)\phi(\b;d)\frac{\theta_q^2}{2}}.
\]
\end{lemma}
\begin{proof}
Fix $x\in \bbZ^d_+, \e>0$ and let $T(\e)= 2^{(1+ \e)\phi(\b;d)
  \frac{\theta_q^2}{2}},\ T^*=2^{2\theta^2_q}.$
Then
\begin{gather}
 \bbE_\o(\t_{x+x_\L})=\int_0^{T(\e)} dt \,\bbP_\o(\t_{x+x_\L}>t) + \int_{T(\e)}^{T^*}dt\,
  \bbP_\o(\t_{x+x_\L}>t) +\int^{+\infty}_{T^*}dt
  \,\bbP_\o(\t_{x+x_\L}>t) \nonumber \\
  \le T(\e) + T^*\bbP_\o(\t_{x+x_\L} >T(\e))
  + \int^{+\infty}_{T^*}dt \,\bbP_\o(\t_{x+x_\L}>t).
  \label{eq:5}
\end{gather}
We will now prove that the supremum over $\o\in \{\o:\,\o_x=0\}$ of the second
and third term in the r.h.s.\ of
\eqref{eq:5} tend to zero as $q\to 0$.
We first need the following general bound whose proof will be provided shortly.
\begin{lemma}
\label{lem:election}
There exist positive constants $c,c'$ independent of $q$ such that the following holds. Fix $\ell\in \bbN$ and for $x\in \bbZ^d_+$ write  $V_{x,\ell}=\{x_1-\ell,\dots,x_1\}\times \dots
\times \{x_d-\ell,\dots,x_d\}\cap\bbZ^d_+$. Then for any box $\L$ with side lengths $(L_1,\dots,L_d)$ and any $t>0$ it holds that
\begin{equation}
\label{eq:24}
\sup_{\o:\ \o_x=0}\bbP_\o(\t_{x+x_\L}>t) \le c't\ell^d e^{-c q \ell}+ 2^{\theta_q(\ell+\max_iL_i)^d -
    t \ell^{-d} \min_{y\in V_{x,\ell}} \l^D(\L_y)},
\end{equation}
where $\L_y= \{y_1,\dots, x_1+L_1\}\times\dots\times
 \{y_d,\dots,x_d+L_d\}.$
\end{lemma}
\begin{remark}
The length scale $\ell$ in the lemma is a free parameter that in the applications we will suitably choose depending on $x,t,\L.$
\end{remark}
Consider now the second term in the r.h.s. of \eqref{eq:5}. In this case we apply Lemma \ref{lem:election} with  $t= T(\e)$ and $\ell=\lfloor \frac 12 \min_i L_i\rfloor$ to bound from above $\bbP_\o(\t_{x+x_\L}>T(\e))$ 
The assumption $\min_i L_i=\Theta\big(2^{\theta_q^{3/2}}\big)$
and the choice of $\ell$ imply that the first term in the r.h.s. of \eqref{eq:24} after multiplication by $T^*$ is $o(1)$ as $q\to 0$. Moreover, the
 fact that $\L$ is $(\b,\k)$-outstretched implies that $\L+y$ is
 $(\b,\k+1)$-outstretched for any $y\in V_{x,\ell}$. In particular, for all $q$ small enough depending only on $\e,\b,\k,$ and for any $y\in V_{x,\ell}$
 \begin{equation}
    \label{eq:8bis}
  \l^D(\L_y)\ge 2^{-(1+\e/2)\phi(\b;d)\frac{\theta_q^2}{2}}.
  \end{equation}
Hence, as
    $q\to 0$
$$
T^*\times(\text{the second term in the r.h.s. of \eqref{eq:24}})\le 2^{2\theta_q^2 + \theta_q 2^{O(\theta_q^{3/2})}} e^{-2^{\e\phi(\b;d)
    \frac{\theta_q^2}{4}}} = o(1).$$ 
We finally consider the third term in the r.h.s.\ of \eqref{eq:5}. In this case, for any $t>T^*$ we apply \eqref{lem:election} with $\ell=\ell_t=t^{1/4d}$. Observe that for some $y\in V_{x,\ell}$ the  box $\L_y$ could be extremely outstretched
in some direction preventing us from using 
Proposition \ref{prop:1}. Hence we are forced to use the spectral gap bound
\eqref{eq:2}
\[
  \min_{y\in V_{x,\ell}} \l^D(\L_y)\ge 2^{-(1+\e)\frac{\theta_q^2}{2}}
  \]
to get that for any $t\ge T^*$ 
  \begin{align*}
\bbP_\o(\t_{x+x_\L}>t) &\le c't\ell^d_te^{-cq t^{1/4d}} +
e^{O(\theta_q)t^{1/4} - t^{3/4}2^{-\frac{\theta_q^2}{2}(1+\e)}}\\
&\le c't^{5/4}e^{-cq t^{1/4d}} + e^{- t^{3/4} 2^{-\frac{\theta_q^2}{2}(1+\e)}/2}.
      \end{align*}
     It now suffices to observe that
      \[
      \int_{T^*}^{+\infty} dt\ \big[c't^{5/4}e^{-cq t^{1/4d}} + e^{- t^{3/4}
        2^{-\frac{\theta_q^2}{2}(1+\e)}/2}\big]=o(1) \quad \text{as $q\to 0$}.
      \]     
\end{proof}
\begin{proof}[Proof of Lemma \ref{lem:election}]
Given $\ell\in \bbN$ and $x\in \bbZ^d_+$ let $\cG(t,\ell), t>0,$ be the event
that there exists $z\in V_{x,\ell}$ such that
\[
  \cT_t(z)=\int_0^t ds\,
  1_{\{c_z(\o(s))=1\}} >t/\ell^d.
  \]
  In other words $z$ is unconstrained for a fraction $\ell^{-d}$ of the time
  $t$. When such a vertex exists we will write $\xi\in V_{x,\ell}$ for the
  smallest  one in the lexicographical order. In
  \cite[Corollary 4.2]{CFM3} it has been proved that
there exist  constants $c,c'>0$ such that
\begin{equation}
  \label{eq:15}
   \sup_{\o\in \{\o:\,\o_x=0\}}\bbP_\o(\cG(t,\ell)^c)\le c' t \ell^d e^{-cq \ell}.
\end{equation}
\begin{remark}
If $t$ is so large that
$V_{x,\ell_t}$ coincides with the box of side lengths $(x_1,\dots,x_d)$,
then the event $\cG(t,\ell_t)^c=\emptyset$ because the origin is always unconstrained.
\end{remark}
Thus, for any $\o$ such that $\o_x=0$, 
\[
\bbP_\o(\t_{x+\L}>t)\le c' t \ell^d e^{-cq \ell} + 
 \bbP_\o(\t_{x+x_\L}>T(\e);\ \cG(x,\ell)).
 \]
Recall that $\L_y= \{y_1,\dots, x_1+L_1\}\times\dots\times
 \{y_d,\dots,x_d+L_d\}$ and 
let  
$\cF_{y,t}$ be the $\s$-algebra generated by the variables
$\{\o_z(s): z\in \partial_{\downarrow}\L_y, s\le t\}$. Notice that $\{c_y(\o(s))\}_{s\le t}$ is measurable w.r.t.\ $\cF_{y,t}$ so that 
\begin{align*}
  \bbP_\o(\t_{x+x_\L}>t;\ \cG(x,\ell))&= \sum_{y\in V_{x,\ell}}\bbP_\o(\t_{x+x_\L}>t;\xi=y)\\
  &=  \bbE_\o
(1_{\{\xi=y\}}\bbP_\o(\t_{x+x_\L}>t \tc
\cF_{y,t})).
\end{align*}
The orientation of the East process implies that, conditionally on $\cF_{y,t},$ the event $\{\t_{x+x_\L}>t\}$ coincides with the same
event for the  \emph{time-inhomogeneous} East chain in $\O_{\L_y}$
with \emph{deterministic, time-dependent}
boundary conditions on $\partial_{\downarrow} \L_y$. We denote the law of the latter chain with initial
state $\o\restriction_{\L_y}$ by $\hat \bbP_{\o}(\cdot)$. Thus,
\begin{align}
  \label{eq:7}
\bbP_\o(\t_{x+x_\L}>t\tc
  \cF_{y,t}) &= \hat \bbP_{\o}(\t_{x+x_\L}>t)\nonumber \\
  &\le \mu(\o\restriction_{\L_y})^{-1}\sum_{\eta\in
  \O_{\L_y}}\mu(\eta)\hat \bbP_{\eta}(\t_{x+x_\L}>t)\nonumber\\
           &\le 2^{\theta_q |\L_y|}\sum_{\eta\in
  \O_{\L_y}}\mu(\eta)\hat \bbP_{\eta}(\t_{x+x_\L}>t).
\end{align}
Let now $t_0\equiv 0<t_1<t_2<\dots< t_n< t_{n+1}\equiv t$ be the times at which the
boundary conditions  on $\partial_{\downarrow} \L_y$ change and let $\s^{(i)}$ denote
the boundary condition during the time interval $(t_{i-1},t_{i})$.  Let
also $\hat \cL^{(i)}$ be the generator of the East chain on $\O_{\L_y}$ with
boundary conditions $\s^{(i)}$ and
let $\cA^{(i)}= 1_{A^c}\hat \cL^{(i)} 1_{A^c}$ be the generator $\hat\cL^{(i)}$ with 
Dirichlet boundary condition on $A=\{\eta\in \O_{\L_y}:\ \eta_{x+x_\L}=0\}$.
Then,
\[
\sum_{\eta\in
  \O_{\L_y}}\mu_{\L_y}(\eta)  \hat \bbP_{\eta}(\t_{x+x_\L}>t)= \langle \mathbf 1,
  e^{t_1\cA^{(1)}}\times e^{(t_2-t_1)\cA^{(2)}}\times\dots\times
  e^{(t_{n+1}-t_n)\cA^{(n+1)}} \mathbf 1\rangle,
\]
where $\mathbf 1(\eta)=1 \ \forall \eta \in \O_{\L_y}$
and $\langle\cdot,\cdot\rangle$ denotes the scalar product in $\ell^2(\O_{\L_y},\mu_{\L_y})$.
Let $\l_i \ge 0$ be the smallest eigenvalue of
$-\cA^{(i)}$. Clearly, 
\begin{equation}
  \label{eq:8}
\langle \mathbf 1,
  e^{t_1\cA^{(1)}}\times e^{(t_2-t_1)\cA^{(2)}}\times\dots\times
  e^{(t_{n+1}-t_n)\cA^{(n+1)}} \mathbf 1\rangle \leq    e^{-\sum_{i=1}^{n+1}(t_i-t_{i-1})\l_i}.
\end{equation}
If during the time interval $(t_i,t_{i+1})$ the constraint $c_y$ at
the vertex $y$ is zero then we simply use $\l_i\ge 0$. If instead
$c_y=1$ we use monotonicity of $\l_i$ in the boundary conditions $\s^{(i)}$ to write $\l_i\ge \l^D(\L_y)$.
Thus, recalling that
$
  \int_0^t ds\,1_{\{c_y=1\}}\ge t\ell^{-d},
$
we get
\begin{align*}
  &\langle \mathbf 1,
  e^{t_1\cA^{(1)}}\times e^{(t_2-t_1)\cA^{(2)}}\times\dots\times
  e^{(t_{n+1}-t_n)\cA^{(n+1)}} \mathbf 1\rangle \nonumber \\
  &\leq e^{-\l^D(\L_y)\, \int_0^t ds\,1_{\{c_y=1\}}}= e^{-t\ell^{-d} \l^D(\L_y)}.
\end{align*} 
In conclusion, 
\[
 \bbP_\o(\t_{x+x_\L}>t;\ \cG(x,\ell))\le 2^{\theta_q|\L_y|-t\ell^{-d} \l^D(\L_y)}\le 
 2^{\theta_q(\ell +\max_i L_i)^d -t\ell^{-d} \l^D(\L_y)}
 \]
and the statement of the lemma follows.
\end{proof}

\subsection{A bottleneck on scale \texorpdfstring{$2^{\frac{\theta_q}{d}}$}{2t/d}}
\label{sec:bottleneck}

\begin{definition}[Legal updates and legal path]
Consider $\L\subseteq \bbZ^d_+$ together with a boundary condition $\s$ for $\L$ if $\L\neq \bbZ^d_+$. Given $\o\in \O_\L$ and $x\in \L$ we say that the update $\o\to \o^x$ is $\s$-\emph{legal} iff $c^{\L,\s}_x(\o)=1$. A sequence $(\o^{(1)},\dots, \o^{(n)})$ of configurations in $\O_\L$ such that $\o^{(i+1)}$ is obtained from $\o^{(i)}$ by
means of a (non-trivial) $\s$-legal update will be
referred to as a $\s$-\emph{legal path} in $\O_\L$ joining $\o^{(1)}$ to
$\o^{(n)}$. When $\L=\bbZ^d_+$ and $\s$ is missing we will simply write legal update and legal path.  
\end{definition}
Before discussing the core of this section, we point out the following \emph{monotonicity property of legal
updates}. Take two sets $\L\subset \L'\subset \bbZ^d_+$ together with
two boundary conditions $\s,\s'$ on $\partial_{\downarrow}\L$ and $\partial_{\downarrow}\L'$
respectively such that $\s_x=0\ \forall \, x\in \partial_{\downarrow}\L\cap \L'$
and $\s_x\le \s'_x\ \forall \, x\in \partial_{\downarrow}\L\cap \partial_{\downarrow}\L'.$  Then any $\s'$-legal
update inside $\L$ is also a $\s$-legal update.
\begin{definition}[Bottleneck]
\label{def:bottleneck}Let $\L_L=\{0,\dots,L\}^d,$ and for $x\in
\bbZ^d_+\setminus \L_L$ let $V_{x,L}=(\L_L+x-x_{\L_L}) \cap \bbZ^d_+$. We say
that $A \subset \Omega_{V_{x,L}}$ is an
$(x,L)$-\emph{bottleneck} if any legal path in $\O$ joining $E_{x,L}\equiv \{\o\in
\O: \ \o\restriction_{V_{x,L}}=1\}$
with $\{\o: \o_x=0\}$ hits $\{\o:
\o\restriction_{V_{x,L}}\in A\}$. 
  \end{definition}
\begin{proposition}\label{prop:bottleneck_general}
In the setting of Definition \ref{def:bottleneck} for any $\e>0$ there exists $q(\e)>0$ such that for
$q\le q(\e)$ the following holds. For
any $L\le 2^{\theta_q/d}$ and $x\in \bbZ^d_+\setminus \L_L$ there exists a $(x,L)$-bottleneck $A$
with $\mu(A)\le
    2^{-(n\theta_q-d\binom{n}{2})(1-\e)}$ where $n:=\lfloor\log_2(L)\rfloor$.
\end{proposition}
\begin{proof}
Fix $\e>0, L\le 2^{\theta_q/d}$ and $x\in \bbZ^d_+\setminus \L_L,$ and w.l.o.g. suppose that $V_{x,L}\subset
\bbZ^d_+$. The case when this assumption fails follows immediately from
the monotonicity property of legal updates described above. Fix a legal path $\G=(\o^{(1)},\dots, \o^{(k)})$ in $\O$ such that
$\o^{(1)}\in E_{x,L}$ and $\o_x^{(k)}=0$. Finally, write $\o^{(j)}_V$ for the restriction to $V_{x,L}$ of $\o^{(j)}$ and let $1\le j_1<j_2<\dots <j_m\le
k$ be those indices 
such that the legal update connecting $\o^{(j_i)}$ to  $\o^{(j_i+1)}$ occurs inside $V_{x,L}$. Let  $\s_{\rm max}$ denotes the \emph{maximal} boundary condition for $V_{x,L}$. Using the monotonicity of legal
updates, the sequence
$\hat \G = (\o^{(j_1)}_V,\dots, \o^{(j_m)}_V) $ is a $\s_{\rm max}$-legal path in $\O_{V_{x,L}}$ connecting the configuration in 
$\O_{V_{x,L}}$ with no vacancies to $\{\o\in \O_{V_{x,L}}:\ \o_x=0\}$.      
The results of \cite[Section 4]{CFM2} imply that $\hat
\G$ must hit a fixed subset $A$ of $\O_{V_{x,L}}$ (called $\partial A_*$ there) whose
equilibrium probability satisfies the required bound.
\end{proof}
\begin{corollary}
  \label{cor:2}
In the same setting 
  \[
    \max_{\o\in E_{x,L}}\bbP_\o(\t_x< t)\le O(t)\times
    2^{-(n\theta_q-d\binom{n}{2})(1-\e)}.
\]
Notice that for $L=2^{\theta_q/d}$ the r.h.s.\ above becomes equal to $O(t)\times   2^{-\frac{\theta_q^2}{2d}(1-\epsilon)}.$ 
  \end{corollary}
  \begin{proof}
 We only give a quick sketch because the proof of similar statements
 has already appeared elsewhere (see e.g.\ \cite{CFM}). Fix $L\le 2^{\theta_q/d}$ and $x\in
 \bbZ^d_+\setminus \L_L.$ Using Proposition \ref{prop:bottleneck_general} there
 exists $A\subset \O_{V_{x,L}}$ such that
 \[
  \max_{\o\in E_{x,L}}\bbP_\o(\t_x< t)\le   \max_{\o\in
    E_{x,L}}\bbP_\o(\t_{A}\le t).
\]
 For a given $\o\in \O_{V_{x,L}^c}$ write
$\d_\o\otimes \mu_{V_{x,L}}$ for the product measure on $\O$ whose marginals
on $\O_{V_{x,L}^c}\otimes \O_{V_{x,L}}$ are the Dirac mass at $\o$ and $\mu_{V_{x,L}}$
respectively. Using $L\le 2^{\theta_q/d}$ we get that
$\mu_{V_{x,L}}(\o\restriction_{V_{x,L}}=1)^{-1}=O(1)$ as $q\to 0$. Hence,
\begin{align*}
   \max_{\o\in E_{x,L}}\bbP_\o(\t_A\le t)& \le O(1)\times
   \max_{\o\in \O_{V^c_{x,L}}} \bbP_{\d_\o\otimes
                                           \mu_{V_{x,L}}}(\t_A\le t)\\
 &\le O(t\, L^d) \max_{\o\in \O_{V_{x,L}^c}} \sup_{s\le t}\bbP_{\d_\o\otimes \mu_{V_{x,L}}}(\o(s)\restriction_{V_{x,L}}\in  A ).
\end{align*}
It is easy to check (see \cite[Section 3]{CFM3}) that $\mu_{V_{x,L}}$ is stationary for the marginal on $\O_{V_{x,L}}$
of the East
process with initial distribution $\d_\o\otimes \mu_{V_{x,L}}$. Hence,
the r.h.s.\ above is equal to $O(tL^d)\mu(A)\le
O(t)2^{-(n\theta_q-d\binom{n}{2})(1-2\e)}$ for $q$ small enough
depending on $\e$.
\end{proof} 
\section{Proof of Theorems \ref{thm:1}, \ref{thm:2}, and \ref{thm:3}}
\label{sec:proofs}
\subsection{Proof of Theorem \ref{thm:1}: (A)} 
\label{sec: proof thm1}In the sequel $\mathbf x\in \bbR^d_+$
will denote a unit vector independent of $q$ with
$\min_i \mathbf x_i>0.$
\subsubsection{Lower bound on \texorpdfstring{$v_{\rm min}(\mathbf x)$}{vmin(x)}.}
\label{sec:lower bound vmin}Let $\ell=\lfloor
2^{\theta_q^{3/2}}\rfloor$ and let $x^{(n)}=\lfloor n\ell \mathbf x\rfloor
, \ n\in \bbN$. We begin by proving that
\begin{equation}
  \label{eq:9}
      \limsup_{n\to \infty}\frac{\bbE_{\o^*}(\t_{x^{(n)}})}{n}\le
2^{\frac{\theta_q^2}{2d} (1+o(1))}\quad \text{as $q\to 0$.}
\end{equation}
Clearly
\[
  \t_{x^{(n+1)}}\le \inf\{s\ge \t_{x^{(n)}}:\ \o_{x^{(n+1)}}(s)=0\},
  \]
 so that, using the strong Markov property, 
  \begin{align*}
    \bbE_{\o^*}(\t_{x^{(n+1)}})\le \bbE_{\o^*}(\t_{x^{(n)}}) +
    \max_{\o\in \{\o:\, \o_{x^{(n)}}=0\}}\bbE_\o(\t_{x^{(n+1)}}).
  \end{align*}
Let $L_i= x_i^{(n+1)}-(x_i^{(n)}+1), i\in [d].$ Clearly the
box with sides length $(L_1,\dots,L_d)$ is $(0,\k)$-outstretched with
$\k=\max_{i,j}\mathbf x_i/\mathbf x_j +1$ and Lemma \ref{lem:1} implies that, uniformly in $n$, for
any $\e>0$ 
\begin{equation}
  \label{eq:9bis}
\max_{\o\in \{\o:\, \o_{x^{(n)}}=0\}}\bbE_\o(\t_{x^{(n+1)}})\le
2^{\frac{\theta_q^2}{2d}(1+\e)},
\end{equation}
for any $q$ sufficiently small depending on $\e$.
Equation \eqref{eq:9} now follows immediately. 

In order to complete the proof of (A) we write
\begin{equation*}
      \bbE_{\o^*}(\t_{n \mathbf x})\le     \bbE_{\o^*}(\t_{x^{(\lfloor
          n/\ell\rfloor)}}) + \max_{\o\in \{\o:\, \o_{x^{(\lfloor
            n/\ell\rfloor)}}=0\}}\bbE_\o(\t_{n \mathbf x}).
\end{equation*}
By using the arguments entering into the proof of Lemma \ref{lem:1}
it is easy to see that $\sup_n \max_{\o\in \{\o:\, \o_{x^{(\lfloor
      n/\ell\rfloor)}}=0\}}\bbE_\o(\t_{n \mathbf x})<+\infty$. Therefore
\begin{equation*}
  \limsup_{n\to \infty}\frac{\bbE_{\o^*}(\t_{n \mathbf x})}{n}\le
  \ell^{-1}2^{\frac{\theta_q^2}{2d}(1+o(1))}= 2^{\frac{\theta_q^2}{2d}(1+o(1))},
\end{equation*}
 because of the choice of $\ell$. In conclusion we have proved that
 $v_{\rm min}(\mathbf x) \ge 2^{-\frac{\theta_q^2}{2d}(1+o(1))}$ as $q\to 0$. \qed

 \subsubsection{Upper bound on \texorpdfstring{$v_{\rm max}(\mathbf x)$}{vmax(x)}.}
For any
   $y\in \bbZ^d_+$ and $n\le \|y\|_1$ let $H_{y,n}=\{z:\ z\prec y,\,
   \|y-z\|_1\le n\}.$
Fix now $y\in \bbZ^d_+$
 with $\|y\|_1\ge \ell_q=\lfloor 2^{\theta_q/d}\rfloor$ and observe that 
if the starting configuration of
the East process on $\bbZ^d_+$ is $\o^*,$ then $\tau_{\partial_{\downarrow} H_{y,\ell_q}}<\tau_y$ a.s.
Hence, for all $\l>0$ the strong Markov property gives
\begin{align}
  \label{eq:10}
  \bbE_{\o^*}(e^{-\l \t_y})&=\bbE_{\o^*}\big(e^{-\l \tau_{\partial_{\downarrow} H_{y,\ell_q}}}\bbE_{\o_{\tau_{\partial_{\downarrow}H_{y,\ell_q}}}}(e^{-\l\t_y})\big)
  \nonumber \\
                        &\le  W(\l)\sum_{z\in \partial_{\downarrow} H_{y,\ell_q}}\bbE_{\o^*}(e^{-\l \t_z}),
  \end{align}
  where    $W(\l):=\sup_{z:\, \|z\|\ge \ell}\max_{\o\in \{\o:\,
    \o\restriction_{H_{z,\ell_q}}=1\}}\bbE_\o(e^{-\l\t_z})$.
Using $|\partial_{\downarrow} W_{y,\ell_q}|\le O(\ell^{d-1})$ we can iterate \eqref{eq:10}
to get  that
  \[
    \bbE_{\o^*}(e^{-\l \t_y})\le 
    \Big(O(\ell^{d-1})W(\l)\Big)^{\lfloor \|y\|_1/\ell\rfloor}.
    \]
    \begin{claim}
 For any $\e>0$ sufficiently small let
 $T(\e)=2^{\frac{\theta_q^2}{2d}(1-\e)}$ and choose $\l=\l(\e,q)=
 \e\theta_q^2 T(\e)^{-1}$. Then $W(\l(\e,q))\le e^{-\O(\e \theta_q^2 )}$ as $q\to 0$. 
    \end{claim}
    \begin{proof}[Proof of the claim]
 Using Corollary \ref{cor:2}, for any $z$ with $\|z\|_1\ge \ell_q$ and
 any $q$ small enough depending on $\e$, we get
 \begin{align*}
\max_{\o\in \{\o:\,
       \o\restriction_{H_{z,\ell_q}}=1\}}\bbE_\o(e^{-\l\t_z}) 
     &\le e^{-\l T(\e)} + \max_{\o\in \{\o:\,
       \o\restriction_{H_{z,\ell_q}}=1\}}\bbP_\o(\t_z\le T(\e)) \\
   &\le e^{-\e \theta_q^2} + O(T(\e))2^{-\frac{\theta_q^2}{2d}(1-\e/2)}= 
  e^{-\O(\e \theta_q^2 )}.
 \end{align*}
\end{proof}
Using $e^{-\l \bbE_{\o^*}(\t_y) }\le \bbE_{\o^*}(e^{-\l \t_y})$ and
choosing $\l$ as in the claim, we finally obtain 
\begin{align}
  \label{eq:10bis}
   \bbE_{\o^*}(\t_y) \ge \O\big(2^{\frac{\theta_q^2}{2d}(1-\e)}\big) \lfloor 2^{-\theta_q/d}\|y\|_1\rfloor.
 \end{align}
In particular, \eqref{eq:10bis} implies that $v_{\rm max}(\mathbf x)\le
 2^{-\frac{\theta_q^2}{2d}(1-o(1))}$ as $q\to 0.$ 
 \qed
 \begin{remark}
   \label{rem:3}
   Exactly the same proof applies to get the following result. For any
   $\e >0$ there exists $q(\e)>0$ and $c(\e)>0$ such that the
   following holds for $q\le q(\e)$. For any
   $y\in \bbZ^d_+$ and $n\le \|y\|_1$ 
   \[
     \max_{\o:\,  \o\restriction_{H(y,n)}=1}\bbP_{\o}(\t_y\le n
     T(\e))\le e^{- c\e \theta_q^2\lfloor
       n2^{-\frac{\theta_q}{d}}\rfloor}.
   \] 
 \end{remark}
 \subsection{Proof of Theorem \ref{thm:1}: (B)}
The proof is identical to that of Section \ref{sec: proof thm1}
with the following modification. The box $\L$ with side lengths $L_i=
x_i^{(n+1)}-(x_i^{(n)}+1), i\in [d],$ is now $(\b,\k+1)$-outstretched because of the assumption
on the direction $x=x(q)$. Using again Lemma \ref{lem:1}
we get the analogue of\ \eqref{eq:9bis}:
\begin{equation}
  \label{eq:16}
\max_{\o\in \{\o:\, \o_{x^{(n)}}=0\}}\bbE_\o(\t_{x^{(n+1)}})\le
2^{\phi(\b;d)\frac{\theta_q^2}{2}(1+\e)}.
\end{equation}
The rest of the argument remains unchanged and the conclusion is that
\[
  \limsup_{n\to \infty}\frac{\bbE_{\o^*}(\t_{nx})}{n}\le
  \ell^{-1}2^{\phi(\b;d)\frac{\theta_q^2}{2}(1+\e)},
\]
i.e. 
\[
 \limsup_{q\to 0}-\frac{1}{\theta_q^2}\log_2(v_{\rm
   min}(x))\le\frac{\phi(\b;d)}{2}<\frac 12
\]
because $\phi(\b;d)<1$ if $\b\in [0,1)$.\qed
\subsection{Proof of Theorem \ref{thm:1}: (C)}
Fix a $q$-dependent unit vector $\mathbf x\in \bbR^2_+$ such that $0<\mathbf x_2\le
\mathbf x_12^{-\theta_q^2\a}$ with $\a>0$. In order to track how a vacancy can
propagate from the origin
to the vertex $\lfloor n \mathbf x\rfloor\in \bbZ^2_+$ we introduce the following construction.

Let $0<\e\ll 1$ and let $L=L(\e,\a,q)=\lfloor
2^{\theta_q^2\a(1-\e/2)}\rfloor$. W.l.o.g. we assume that $q$ is so
small that $L\gg 2^{\theta_q}$.
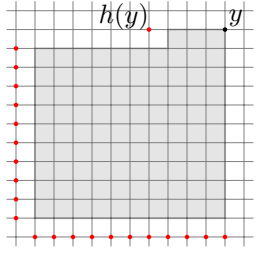
\begin{figure}[H]
    \centering
    \begin{tikzpicture}[scale=0.5]
        \draw[step=0.5cm,gray,very thin] (-0.25,-0.25) grid (6.25,6.25);
        \draw[thick,fill=gray, opacity=0.2] (5.5,5.5) -- ++(-1.5,0) -- ++(0,-0.5) -- ++(-3.5,0)
        -- ++(0,-4.5) -- ++(5,0) -- ++(0,5);
        \draw[thick, cross] (3.5,5.5) node[anchor=south east] {$h(y)$};
        \draw[red,fill=red] (3.5,5.5) circle (.3ex);
        \foreach \x in {0,...,9}{
             \draw[red,fill=red] (0.5+\x*0.5,0) circle (.3ex);
             \draw[red,fill=red] (0,0.5+\x*0.5) circle (.3ex);
        }
        \draw[red,fill=red] (5.5, 0) circle (.3ex);
        \draw[thick, cross] (5.6,5.6) node[anchor=south west] {$y$};
        \draw[black,fill=black] (5.5,5.5) circle (.3ex);
    \end{tikzpicture}
    \caption{Example for a set $U_y$ (the gray region).
    The red vertices denote $\partial_{\downarrow} U_y$.}
    \label{fig:Ux} 
\end{figure}
\begin{definition}\label{def:Ux}
    For $y=(y_1,y_2)\in \Z_+^2$ such that $1\le y_2\le
    2^{-\theta_q^2\a}\,y_1$ let $B_{y,L}\subset \bbZ^2$ be the box
    of side lengths $(L,L)$ and upper-right corner at $y$ and let (see
    Figure \ref{fig:Ux}) 
    \begin{align*}
        U_y = \Big(B_{y,L}\setminus \cup_{i=\lfloor
      1/q\rfloor+1}^{L}\{y-i\mathbf e^{(1)}\} \Big)\cap \Z_+^2.
    \end{align*}
Let also $h(y):=y-(\lfloor 1/q\rfloor +1)\mathbf e^{(1)} $ and
    note that $h(y)\in \partial_{\downarrow} U_y$. 
\end{definition}
If the starting configuration of
the East process on $\bbZ^2_+$ is $\o^*,$ then $\t_{\partial_{\downarrow}U_y}<\t_{U_y}<\t_y$. This
observation justifies the following definition. In the sequel
$\{\o_t\}_{t\ge 0}$ denotes the East process in $\bbZ^2_+$ with
  $\o_0=\o^*$.
\begin{definition}[Infection sequence for $y$]
\label{def:1}Let $\xi^{(0)}=y$ and define
recursively $\xi^{(i)}$ as the unique vertex $z\in
 \partial_{\downarrow} U_{\xi^{(i-1)}}$ such that $\o_{\tau_{\partial_{\downarrow} U_{\xi^{(i-1)}}}}(z)=0$. We also let $\nu:= \inf\{i\in \N\colon 0\in U_{\xi^{(i)}}\}$
and call the random sequence $\xi(y)=\{\xi^{(i)}\}_{i\in [\nu]}$ the \emph{infection sequence 
for $y$}.
The collection of all possible infection sequences is denoted by
$\mathcal{S}(y)$. Given $\mathbf v =\{v^{(i)}\}_{i} \in \mathcal{S}(y)$ we say
that $v^{(i)}$ is \emph{good} if $v^{(i+1)}=h(v^{(i)})$ and \emph{bad}
otherwise.
\end{definition}
\begin{remark}
  \label{rem:2}
  By construction any possible infection sequence $\mathbf v$ is such that
$\|v^{(i)}-v^{(i+1)}\|_1\ge \lfloor 1/q\rfloor$.  
\end{remark}
\begin{lemma}\label{lem:2}
For any $q$ small enough, any infection
sequence in $\mathcal{S}(y)$ contains at most $y_2$ bad points and at
least $\lfloor y_1 \frac q2 \rfloor $ good
points.
\end{lemma}
\begin{proof}
Given an infection sequence $\mathbf v$ let $n_g$ be the
number of its good points and observe that if $v^{(i)}$ is
bad then $v_2^{(i+1)}<
v_2^{(i)}$ and $v_1^{(i)}-
v_1^{(i+1)}\le L$. Hence, $(n-n_g)\le y_2$ and 
\begin{align*}
    (n-n_g) L+n_g/q \ge y_1-L,
\end{align*}
i.e.\ $n_g\ge q(y_1 -L(1+y_2))$. In particular, if $1\le y_2\le
2^{-\theta_q^2 \a}y_1$ then $n_g\ge \lfloor y_1 q/2\rfloor  $ for $q$
small enough. 
\end{proof} 
For any $y\in \bbZ^d_+$ let $n_y=\lfloor y_1 \frac q2 \rfloor $ and for any given $\mathbf v\in \mathcal{S}(y)$ let
$(w^{(1)}, w^{(2)},\dots,w^{(n_y)})$ be the collection of
the first  $n_y$ good points of $\mathbf v$ ordered from the
last one to the first one. By construction, for all $k,$ $w^{(k-1)} \prec h(w^{(k)}).$ Using Definition \ref{def:1}, the event  $\{\xi(y)=\mathbf
v\} $ implies the event
\[
  G_{\mathbf
      v}:=\cap_k\{\tau_{U_{w^{(k)}}} = \tau_{h(w^{(k)})} ;
    \tau_{h(w^{(k)})}\ge \tau_{w^{(k-1)}}\},
  \]
  and $\tau_y\ge
    \sum_k (\tau_{w^{(k)}}-\tau_{h(w^{(k)})}).$
 Therefore, for all $\l>0$ the definition of the event $G_{\mathbf v}$ together with a repeated use
of the
strong Markov property implies that
\begin{align}
  \label{eq:11}
e^{-\l \bbE_{\o_*}(\t_y)} &\le   \bbE_{\o_*}(e^{-\l \t_y})  \le \sum_{\mathbf v \in
    \mathcal{S}(y)}\bbE_{\o_*}(\id_{G_{\mathbf v}}e^{-\l \sum_{k=1}^{n_y} (\tau_{w^{(k)}}-\tau_{h(w^{(k)})})})\nonumber\\
  &\le  |\mathcal{S}(y)| \max_{\mathbf
    v}\bbE_{\o_*}\big(\id_{G_{\mathbf v}}\prod_{k=1}^{n_y}
    e^{-\l (\tau_{w^{(k)}}-\tau_{h(w^{(k)})})}\big)\nonumber\\
  &\le |\mathcal{S}(y)| F(\l)^{n_y},
\end{align}
where $|\mathcal{S}(y)|$ denotes the cardinality of $\mathcal{S}(y)$ and
\begin{equation}
  \label{eq:12}
F(\l):= \max_{z\in \bbZ^2_+:\, h(z)\in
    \bbZ^2_+}\, \max_{\o:\ \o(h(z))=0,\, \o\restriction_{U_z} =1} \bbE_\o\big(e^{-\l \t_z}\big).
\end{equation}
The next two lemmas provide the necessary bounds on $|\mathcal{S}(y)|$
and $F(\l)$.
\begin{lemma}\label{lem:3} For any $y\in \bbZ_+^2$ with $1\le y_2<y_1
  2^{-\a \theta_q^2}$ as $q\to 0,$ we have
    \begin{align}
        |\mathcal{S}(y)| \le \big(y_1/y_2\big)^{O(y_2)}.
    \end{align}
\end{lemma}
\begin{proof}
Recall that a good point of an infection sequence specifies uniquely
the next point of the sequence. Hence, we can reconstruct the full
infection sequence by specifying which points are bad together with their relative position
w.r.t.\ the previous point. Using Remark \ref{rem:2} together with $n_y=\lfloor y_1 \frac q2 \rfloor $, it also follows that the length $n$ of any
infection sequence satisfies $n\in [n_y,q(y_1+y_2)]$. Thus for $q$ small enough
    \begin{align*}
        |\mathcal{S}(y)|
               &\le \sum_{n=n_y}^{\lceil
                 q(y_1+y_2)\rceil} \sum_{m=0}^{y_2} \binom{n}{m}
                 {(2L)}^m \le \sum_{n=n_y}^{\lceil q(y_1+y_2)\rceil}  \binom{n}{y_2} (y_2+1)  (2L)^{y_2}\\
        &\le e^{O(\theta_q^2)y_2}\times O(q)y_1\times \binom{\lceil q(y_1+y_2)\rceil}{y_2}
          \le \big(y_1/y_2\big)^{O(y_2)}.
    \end{align*}
\end{proof}
\begin{lemma}\label{lem:4}Fix $0<\e\ll 1$ and let $T_\a=T_\a(\e,q)=
  2^{\frac{\theta_q^2}{4} ((1+4\a) \wedge 2)(1-2\e)}$. Then for  
any $q$ sufficiently small and any $\l>0$
    \begin{align*}
        F(\l)\le
        e^{-\l T_\a}+ 2^{-\O(\e)\theta_q^2}.
    \end{align*}
\end{lemma}
\begin{proof}
Fix $z\in \bbZ^2_+$ such that $h(z)\in
\bbZ^2_+$ together with $\omega$ such that $\omega(h(z))=0$ and
$\o\restriction_{U_z}=1$. Let also $A:=\{h(z)+\mathbf e^{(1)}-\mathbf e^{(2)},h(z)+2\mathbf e^{(1)}-\mathbf e^{(2)},\dots,z-\mathbf e^{(2)}\}$.
Then,
\begin{align*}
\bbE_{\o}(e^{-\l\tau_z})
  &\le e^{-\l T_\a}+ \bbP_{\o}(\tau_z<  T_\a)\\
  &\le e^{-\l T_\a}+ \bbP_{\o}(\{\tau_z<  T_\a\}\cap \{\t_A>T_\a\})+
    \bbP_{\o}(\t_A\le T_\a)\\
  &\le e^{-\l T_\a}+ \bbP_{\o}(\{\tau_z<  T_\a\}\cap \{\t_A>T_\a\})+
    \sum_{a\in A}\bbP_{\o}(\t_a\le T_\a).
\end{align*}
Let $\cF_{T_\a}$ be the $\s$-algebra generated by the variables $ \o_z(s),
s\in [0,T_\a]$ where $z\in \{a\in \bbZ^2_+: a\prec h(z)\}\cup
\{a\in \bbZ^2_+: a\prec b \text{ for some $b\in A$}\}$. Clearly
$\{\t_A> T_\a\}\in \cF_{T_\a}$. Moreover, conditionally on $\cF_{T_\a}$ and on the
event $\{\t_A> T_\a\},$ the East process on $A+\mathbf e^{(2)}$ 
coincides up to time $T_\a$ with the
one-dimensional East chain on $A+ \mathbf e^{(2)}$ with a boundary value at
$\{\o_{h(w)}(s)\}_{s\le T}$ which is measurable w.r.t.\ $\cF_{T_\a}$. 
We can then apply Corollary \ref{cor:2} with $d=1$ and
$n=\lfloor\theta_q\rfloor$ to obtain:
\begin{gather} \label{eq:14}
  \bbP_{\o}(\{\tau_z<  T_\a\}\cap \{\t_A>T_\a\})\le
                                            O(T_\a)2^{-\frac{\theta_q^2}{2}(1-\e)}\nonumber\\
  = O\big(
  2^{-\frac{\theta_q^2}{4}((2 - (1+4\a) \wedge 2)(1-2\e)+2\varepsilon)}\big)\le 2^{-\O(\e)\theta_q^2}. 
\end{gather}
Let $n_A= \min_{a\in A} \min_{z'\prec a,\ z'\notin U_z }\|a-z'\|_1,$ 
and observe that $\exists\, \e(\a)>0$ such that $\forall\,\e\le
\e(\a)$ and all $q$ small enough depending on $\e,$ 
$
 T_\a \le n_A\, 2^{\frac{\theta_q^2}{4}(1-\e)}.
$
We can then use Remark \ref{rem:3} to get
that
\[
  \sum_{a\in A}\ \max_{\o:\ \o\restriction_{U_z}=1}\bbP_\o(\t_a\le
  T_\a)\le e^{- \O(\e \theta_q^2\lfloor
    n_A2^{-\frac{\theta_q}{2}}\rfloor)}\le 2^{-\O(\e)\theta_q^2},
\]
because $n_A\ge L -2^{\theta_q}\gg 2^{\theta_q/2}$. 
\end{proof}
We can now conclude the proof. By combining the two lemmas above and
choosing $\l =\l_\a(q)=T_\a^{-1}\e \theta_q^2,$ we get from \eqref{eq:11} that
\begin{align*}
   e^{-\l \bbE_{\o^*}(\t_y)}  &\le |\mathcal{S}(y)|
                             F(\l)^{n_y}\le
    \big(y_1/y_2\big)^{O(y_2)}e^{-\O(\e)\theta_q^2 n_y},
\end{align*}
where we recall that $n_y:=\lfloor y_1 \frac q2 \rfloor $. If $y= \lfloor
n\mathbf{x} \rfloor$ with $\mathbf{x}$ such that
$0<x_2\le x_1 2^{-\theta_q^2\a},$ the above inequality implies 
\[
  \bbE_{\o^*}(\t_{\lfloor n\mathbf{x} \rfloor})\ge \O(q\, T_\a)\times n\quad
  \text{as $n\to \infty$}.
\]
In particular $v_{\rm max}(\mathbf x)\le 2^{-\frac{\theta_q^2}{4} ((1+4\a) \wedge 2)(1-o(1))} $.    
\subsection{Proof of Theorem \ref{thm:2}}
We begin with the case $\d=0$.

Recall Remark \ref{rem:1} and 
that $v_{\rm min}(\mathbf e^{(i)})=v_{\rm max}(\mathbf e^{(i)})=
2^{-\frac{\theta_q^2}{2}(1+o(1))}\ \forall \, i\in [d].$ Take $0<\e\ll
1$ and let 
$x_t=\lfloor 2^{-\frac{\theta_q^2}{2d}(1+\e)}\,t\rfloor\, \mathbf e^{(1)},
\ t\gg 0.$ By construction $x_t\in \L(\d=0,\e,t)$. Let also
\[
  A_t=\{\o:\ \exists\, y\in \{x_t-\lfloor 2^{2\theta_q}\rfloor \mathbf e^{(1)},\dots, x_t\} \text{ such that }
  \o_y(t)=0\},
\]
and use
\[
\|\nu_t^{\d,\e} -  \mu_{\L(\d,\e,t)}\|_{TV}\ge |\mu(A_t)-\nu_t^{\d,\e}(A_t)|.
\]
For any $t$ large enough $\mu(A_t)=1-e^{-\O(2^{\theta_q})},$ while Remark
\ref{rem:3} gives
$\limsup_{t\to \infty}\nu_t^{\d,\e}(A_t)=0$. Hence,
\[
  \liminf_{q\to 0}\liminf_{t\to \infty}\|\nu_t^{\d,\e} -
  \mu_{\L(\d,\e,t)}\|_{TV}=1.
\]
We now consider the case $0<\d < 1.$

Fix $0<\e\ll 1$ and observe (see \cite[Lemma 5.5]{CFM3}) that equilibrium in the region $\L(\d,\e,t)$ is
achieved very rapidly, within a time $O(\log(|\L(\d,\e,t)|)^{4d}),$ if the initial
configuration has a vacancy in every
interval of $\L(\d,\e,t) $ parallel to a coordinate direction and containing $O((\log(|\L(\d,\e,t)|)^2)$ vertices. Hence, if the above condition is
satisfied by the East process at time $t/2$ then at time $t$  the measure $\nu_t^{\d,\e} $ will be very
close to $\mu_{\L(\d,\e,t)}$ in the total variation distance. The
second observation 
(cf. \cite[Lemma 5.3]{CFM3}) is the following. Recall that $\t_x$ is the first time a
vacancy appears at $x$. Then the above requirement for the East
process at time $t/2$ will be fulfilled with w.h.p. if $\t_x\le t/2 -
O((\log(|\L(\d,\e,t)|)^2)\ \forall \,
x\in\L(\d,\e,t)$. 

A more precise formulation of the above two steps is as follows.
For any $t$ large
enough depending on $q,\d,\e$
\begin{equation}
 \label{eq:TV} \|\mu_{\L(\d,\e,t)}-\nu_t^{\d,\e}\|_{TV}\le \e +  \sum_{x\in
    \L(\d,\e,t)}\bbP_{\o^*}(\t_x>t/3).   
\end{equation}
We decided to skip the proof of \eqref{eq:TV} as it follows very closely
the proofs of Lemma 5.3. and 5.5 of \cite{CFM3}.
The proof of the theorem then boils down to proving that the second term in the r.h.s.\ of \eqref{eq:TV} vanishes as $t\to \infty$. For future needs we actually prove a slightly stronger result.
\begin{lemma}\label{lemma:LD} For any $\d, \e$ in $(0,1)$
there exists  $ q(\d,\e)>0$ such that for
any $q\le q(\d,\e)$ and all $t$ large enough
\begin{equation}
 \label{eq:LD}
\sup_{y\in \bbZ^d_+} \sum_{x\in \L(\d,\e,t)+y}\  \sup_{\o:\, c_y(\o)=1}\bbP_{\o}\big(\t_x>t/3)\le e^{-\O\big(2^{-(1+\e/2)\frac{\theta_q^2}{2d}}\,\log^2(t)\big)}.
\end{equation}

\end{lemma}
\begin{proof}[Proof of the lemma]
Fix $y\in \bbZ^d_+$ together with $\o$ such that $c_y(\o)=1$. In the sequel all estimates will be uniform in $y,\o$. Fix $x\in \L(\d,\e,t)+y$ and let $\mathbf x= (x-y)/|x-y|$ be the associated unit vector in
$\bbR^d_+$. Clearly the components of $\mathbf x$ satisfy $\min_{i,j}\mathbf x_i/\mathbf x_j\ge
\d$. Let $\ell_q= 2^{\theta_q^{3/2}}$, let
$n_x=\lfloor |x-y|/\ell_q\rfloor,$ and define the sequence of vertices $\{x^{(n)}\}_{n=0}^{n_x+1}$ by
$x^{(n)}=\lfloor n\ell_q  \mathbf x\rfloor$ if $0\le n\le n_x$ and $x^{(n_x+1)}=x$. By construction $
|x^{(n+1)}-x^{(n)}|\le \ell_q+1,$ and 
 $\exists \,\k(\d)\ge 1, q(\d)<1$ such that $\forall \, q\le q(\d)$ 
\[
  \max_{0\le n\le
    n_x}\max_{i,j}\frac{(x^{(n+1)}-x^{(n)})_i}{(x^{(n+1)}-x^{(n)})_j} \le
  \k(\d).
  \]
For the East process with initial condition $\o$ recursively define
  \begin{align*}
    \t^{(0)} =\inf\{s\ge 0, \o_{x^{(0)}}(s)=0\},
    \quad \t^{(n)}=\inf\{s\ge \t^{(n-1)}:
              \ \o_{x^{(n)}}(s)=0\}, \end{align*}
and set $\D_n=\t^{(n)}-\t^{(n-1)}$. 
Finally, let $M= \log(t)^{5d} \times 2^{\frac{\theta_q^2}{2d}(1+\e/2)}.$  Using $\t_x\le
\sum_{n=1}^{n_x+1}\D_n$ we write 
\begin{gather}
\bbP_{\o}\big(\t_x\ge t/3\big)\nonumber\\\le 
\bbP_{\o}\big(\sum_{n=1}^{n_x+1}\D_n \id_{\{\D_n\le M\}}\ge t/3\big)
 +
\sum_{n=1}^{n_x+1}   \sup_{\o:
    \, \o_{x^{(n-1)}}=0}\bbP_{\o}\big(\D_n\ge M\big).
\label{eq:88}
\end{gather}
In order to bound from above the second term in \eqref{eq:88} we
apply Lemma \ref{lem:election} to $x= x^{(n-1)},$ $\L$ the box with sides $L_i= x^{(n)}_i-x^{(n-1)}_i$, $t=M$, and $\ell=\ell_t=\log^2(t)$ to get \[
\sup_{\o:
    \, \o_{x^{(n-1)}}=0}\bbP_{\o}\big(\D_n\ge M\big) \le c' M \ell_t^d e^{-cq \ell_t} + 2^{\theta_q
  (\ell_t+\ell_q +1)^d - M \ell_t^{-d} 2^{-\frac{\theta_q^2}{2}(1+\e)}}.
\]
Using $M\ell_t^{-d}=\O(\log(t)^{3d})$ as $t\to +\infty$, for any $t$ large enough depending on $q$ the second term in
the r.h.s.\ of \eqref{eq:88} satisfies
\begin{equation}
  \label{eq:25}
  \sum_{n=1}^{n_x+1}   \sup_{\o:
    \, \o_{x^{(n-1)}}=0}\bbP_{\o}\big(\D_n\ge M\big) \le e^{-\O(q \log^2(t))}. 
\end{equation}
We now tackle the first term in the r.h.s.\ of \eqref{eq:88} via the
exponential Chebyshev inequality with $\l=
2^{-\frac{\theta_q^2}{2d}(1+\e/2)}\, \log^2(t)/t 
$. Using the
strong Markov property and $\l M\le 1$ for any large enough $t$ we obtain
 \begin{gather*}
   \bbP_{\o}\big(\sum_{n=1}^{n_x+1}\D_n \id_{\{\D_n\le M\}}\ge t/3\big)
   \le e^{-\l t/3}\times \bbE_{\o}\big(\prod_{n=1}^{n_x+1} e^{\l \D_n \id_{\{\D_n\le M\}}} \big)\\
  \le e^{-\l t/3}\times \Big(\sup_n\sup_{\o:
    \, \o_{x^{(n-1)}=0}}\bbE_\o\big(e^{\l \D_n \id_{\{\D_n\le M\}}}
 \big )\Big)^{n_x+1}\\
  \le e^{-\l t/3}\times \Big(1+ e\l \sup_n\sup_{\o:
    \, \o_{x^{(n-1)}}=0}\bbE_\o\big(\D_n\big)\Big)^{n_x+1},
   \end{gather*}
where we used  $e^a\le 1+e
 a, \ \forall \, 0\le a\le 1$ in the last inequality. We can finally
 appeal to Lemma \ref{lem:1} to get that for all $q$ small enough
 depending on $\d,\e$
 \begin{gather*}
 1+ e \l \sup_n\sup_{\o:
    \, \o_{x^{(n-1)}}=0}\bbE_\o\big(\D_n\big)
 \le 1+ e \l\, 2^{(1+\e/2)\frac{\theta_q^2}{2d}} \le  e^{e\log^2(t)/t}.
 \end{gather*}
In conclusion,
 \begin{gather}
\label{eq:1000} \bbP_{\o}\big(\sum_{n=1}^{n_x+1}\D_n \id_{\{\D_n\le M\}}\ge t/3\big) 
     \le e^{-\l t/3 + e (n_x+1)\log^2(t)/t }\le e^{-\l t/6},
 \end{gather}
where we used $(n_x+1) \le |x-y|+1 \le t\,
2^{-\frac{\theta_q^2}{2d}(1+\e)}+1$ to obtain the last inequality for
$q$ small enough depending on $\e$. The claim of the lemma now follows from \eqref{eq:88},\eqref{eq:25} and \eqref{eq:1000}. \end{proof}
\subsection{Proof of \texorpdfstring{\cref{thm:3}}{Theorem 3}}
Using Remark \ref{rem:1} $d(t)\ge \bar d(t),$ where $\bar d(t)$ is defined as $d(t)$ but for the \emph{one dimensional} East chain on $\{0,\dots,n\}$. Hence \eqref{eq:cutoff1} follows directly from the cutoff result for the latter chain (see \cite[Theorem 2]{Ganguly-Lubetzky-Martinelli}).
We now turn to the proof of \eqref{eq:cutoff2}. 

Let $w_n= n^{2/3}$ and let $\hat T_n= T_n + w_n/2$. As in the proof of \cref{thm:2} (see \eqref{eq:TV} and the explanation immediately before) the following can be proved by following very closely the proof of Lemma 5.3 and Lemma 5.5 of \cite{CFM3}.
\begin{lemma}
\label{lem:cutoff}
For any $q\in (0,1)$ 
\begin{equation}
\label{eq:cutoff4}
\limsup_{n\to \infty} d(T_n+w_n)\le \limsup_{n\to \infty}\max_{\o\in \O_{\L_n}}\bbP_\o(\exists\, x\in \L_n: \ \t_x \ge \hat T_n).    
\end{equation}
\end{lemma}
We will now prove that for $q$ small enough 
\begin{equation}
    \label{eq:cutoff4bis}
 \limsup_{n\to \infty}\max_{\o\in \O_{\L_n}}\sum_{x\in \L_n}\bbP_\o(\t_x \ge \hat T_n)=0.
\end{equation} 
We will give the full details for $d=2$ and only sketch the additional  steps needed for $d\ge 3$. In the sequel $\e$ will be a small positive constant, $q$ will be assumed to be sufficiently small depending on $\e, $ and $c(q)$ will denote a positive constant depending on $q$ whose value may change from line to line.   

The intuition behind \eqref{eq:cutoff4bis} is as follows. Fix $x\in \L_n$ and w.l.o.g. suppose that $x_1= \max(x_1,x_2)$. Then the infection time $\t_x$ should be dominated by the sum of the infection time of the vertex $ x'=(x_1-x_2,0)$ plus the infection time of $x$ starting from $\o_{\t_{x'}}$. Using \cite[Theorem 2]{Ganguly-Lubetzky-Martinelli} the first time is, with great accuracy, $(x_1-x_2)/v,$ while part (A) of \cref{thm:1} suggests that w.h.p. the second time is $O\big(x_2/v_{\rm min}(\hat {\mathbf e})\big)$ where $\hat{\mathbf e}= (\frac{1}{\sqrt{2}},\frac{1}{\sqrt{2}})$. Hence, we expect $\t_x$ to satisfy w.h.p. 
\[
\t_x \lesssim (x_1-x_2)/v+ x_2/v_{\rm min}(\hat{\mathbf e})\lesssim n/v \quad \forall x\in \L_n,
\] 
because $v_{\rm min}(\hat{\mathbf e})\gg v$ for $q$ small enough. In other words, the time needed to infect all vertices of $\L_n$ should be dominated by the time needed to infect at least once all vertices of the form $x=(j,0)$ or $x=(0,j), j\in \{0,\dots,n\}.$ In turn,  using the one dimensional cutoff result, the latter time is smaller than $\hat T_n$ w.h.p.

We will now detail the intuition above. We cover $\L_n$ with two regions:
\begin{align*}
\L_n^{(1)}&=\{x\in\L_n: \max_{i}x_i\le \log(n)^4 \},\\
 \L_n^{(2)}&=\{x\in \L_n: \max_{i}x_i\ge \log(n)^4\}, 
\end{align*}
and we will prove that
\begin{equation}
\label{eq:cutoff5}
\limsup_{n\to \infty}\max_{\o\in \O_{\L_n}}\sum_{x\in \L^{(i)}_n}\bbP_\o(\t_x \ge \hat T_n)=0, \ \forall \, i\in [2].
\end{equation}
\begin{enumerate}
\item[$(i=1)$] W.l.o.g. fix $x \in \L_n^{(1)}$ with $x_2\le x_1$ and write 
$\hat x$ for the vertex $(x_1,0)$. Using the strong Markov property we get 
\begin{gather*}
\max_\o\bbP_\o(\t_x\ge \hat T_n)
\le \max_\o\bbP_\o(\t_{\hat x}\ge \hat T_n - w_n/4) + \max_{\o: \, \o_{\hat x}=0}\bbP_{\o}(\t_x>w_n/4). 
\end{gather*}
Using once again \cite[Theorem 2]{Ganguly-Lubetzky-Martinelli} \[
\limsup_{n\to \infty}\sum_{x\in \L_n^{(1)}}\max_\o\bbP_\o(\t_{\hat x}\ge \hat T_n - w_n/4)=0.
\]
Notice that $\|x-\hat x\|_1=2x_2\le \log(n)^4\ll w^{3/8}_n$. Hence, the term $\max_{\o: \, \o_{\hat x}=0}\bbP_{\o}(\t_x>w_n/4)$ can be bounded from above by applying Lemma \ref{lem:election} with $\L=\{0\}\times \{x_2\}$, the vertex $x$ equal to $\hat x$, $t=w_n/4$ and e.g. $\ell=w_n^{1/4}$. Using \eqref{eq:24} for any $n$ large enough we get 
\[
\max_{\o: \, \o_{\hat x}=0}\bbP_{\o}(\t_x>w_n/4)\le e^{-c(q)w_n^{1/4}},
\]
so that 
\[
\limsup_{n\to \infty}\sum_{x\in \L_n^{(1)}}
\max_{\o: \, \o_{\hat x}=0}\bbP_{\o}(\t_x>w_n/4)=0.
\]
\item[$(i=2)$] Fix $x\in \L_n^{(2)}$ with e.g.\ $x_2\le x_1$ and $x_1\ge \log(n)^4$. We can assume further that $x_2/x_1 \le 1/2$ since otherwise  $\max_\o\bbP_\o(\t_x\ge \hat T_n)$ could be bounded from above using Lemma \ref{lemma:LD} to get $\max_\o\bbP(\t_x\ge \hat T_n)\le e^{-c(q)\log(n)^2}$. If $x_2=0$ we can simply apply \cite[Theorem 2]{Ganguly-Lubetzky-Martinelli} to get $\max_\o\bbP(\t_x\ge \hat T_n)\le e^{-c(q) n^{1/3}}$ for some constant $c(q)>0$. Otherwise, let $\phi(x)=x_1-x_2$ and set now $\hat x=(\phi(x)-1,0)$. By construction, the direction of the vector $x-(\hat x+\mathbf{e}^{(1)})$ is the $(\frac{1}{\sqrt{2}},\frac{1}{\sqrt{2}})$-direction. Let also $\varphi_n(x)= \max\big(\frac{\phi(x)}{v} + \frac{\phi(x)^{2/3}}{4}, w_n/5\big).$ 
As in the previous step we write 
\begin{align}
\label{eq:cutoff6}&\max_\o\bbP_\o\big(\t_x\ge \hat T_n\big)\nonumber\\\le \max_\o\bbP_\o\big(\t_{\hat x}\ge \varphi_n(x)\big) &+ \max_{\o: \, \o_{\hat x}=0}\bbP_{\o}\big(\t_x>\hat T_n-\varphi_n(x) \big). 
\end{align}
Using \cite[Theorem 2]{Ganguly-Lubetzky-Martinelli}) applied to the interval $\{0,\dots,\phi(x)\}$ we get that the first term in the r.h.s.\ of \eqref{eq:cutoff6} is bounded from above by   $e^{-c(q)w_n^{1/3}}$ for large $n,$ so that 
\[
\limsup_{n\to \infty}\sum_{x\in\L_n^{(3)}}\max_\o\bbP_\o\big(\t_{\hat x}\ge \varphi_n(x)\big) =0.
\]
For the second term in the r.h.s. of \eqref{eq:cutoff6} we crucially observe that
\[
\hat T_n-\varphi_n(x)\ge
\begin{cases}
\frac{w_n}{4} + \frac{x_2}{v} & \text{ if $\frac{\phi(x)}{v} + \frac{\phi(x)^{2/3}}{4}\ge w_n/5 ,$}\\  T_n+\frac{3}{10}w_n &\text{ otherwise.} 
\end{cases} 
\]
In both cases, using $v\le 2^{-\frac{\theta_q^2}{2}(1-\e)},$ we get that 
$\hat T_n-\varphi_n(x)\gg 2^{\frac{\theta_q^2}{4}(1+\e)}\|x-\hat x\|_1.
$ 
Hence, we can apply Lemma \ref{lemma:LD} with 
$y=\hat x +\mathbf e^{(1)}$, $\d=\frac 13$, and $t=3\big(\hat T_n-\varphi_n(x)\big)$
to get that
\begin{align*}
\max_{\o: \, \o_{\hat x}=0}\bbP_{\o}\big(\t_x>\hat T_n-\varphi_n(x)\big)&\le e^{-\O\big(2^{-(1+\e/2)\frac{\theta_q^2}{4}}\,\log^2(\hat T_n-\varphi_n(x)\big)}\\
&\le e^{-c(q) \log(w_n)^2}.    
\end{align*}
In conclusion
\[
\limsup_{n\to \infty}\sum_{x\in\L_n^{(2)}}\max_{\o: \, \o_{\hat x}=0}\bbP_{\o}\big(\t_x>\hat T_n-\varphi_n(x)\big) =0.
\]
\end{enumerate}

We will now briefly discuss the proof of \eqref{eq:cutoff4bis} when $d\ge 3$. 
The proof of \eqref{eq:cutoff5} for $i=1$ does not change. 
The proof for $i=2$ needs instead a few changes. 

Fix $x\in \L_n^{(2)}$ and w.l.o.g. assume that $1\le x_d\le x_{d-1}\le\dots\le x_1$. For $k\in [d-1]$ define recursively 
\[
\hat x^{(0)}=x,\quad \hat x^{(k)}=\hat x^{(k-1)}- \hat x^{(k-1)}_{d-k+1}\sum^{d-k+1}_{j=1}\mathbf e^{(j)},
\]
so that $\hat x^{(k)}_j=x_j-x_{d-k+1}$ if $j<d-k+1 $ and $\hat x^{(k)}_j=0$ otherwise.
Notice that the direction vector $\mathbf w^{(k)}$ corresponding to each $x^{(k-1)}-x^{(k)}$ when the latter is non-zero has the form $\mathbf w^{(k)}= \sum_{j=1}^{d-k+1}\mathbf e^{(j)}$. Hence, using Remark \ref{rem:1} and part (A) of \cref{thm:1}, the corresponding minimal velocity $v_{\min}(\mathbf w^{(k)})$ satisfies $v_{\min}(\mathbf w^{(k)})\ge  2^{-\frac{\theta^2_q}{4}(1+o(1))}\gg v$. Let also
\[
\varphi^{(k-1)}_n(x)=
\begin{cases}
\max\big(2^{\frac{\theta_q^2}{4}(1+\e)}\|x^{(k-1)}-x^{(k)}\|_1, \frac{w_n}{5d}\big) & \text{ if $k\le d-1$},\\
\max\big(\frac{x_1^{(d-1)}}{v} + \frac{(x_1^{(d-1)})^{2/3}}{4}, \frac{w_n}{5d}\big) & \text{ if $k=d$}.
\end{cases}
\]
Using $v\ll 2^{-\frac{\theta^2_q}{4}(1+o(1))}$ for $q$ small enough, it is easy to check that
\[
\sum_{k=1}^{d} \varphi^{(k-1)}_n(x)
\le \frac{9}{20} w_n + \frac{x_1}{v}\le \hat T_n.    
\]
Hence, by setting recursively
$\s_d= 0$ and $\s_{k-1}=\inf\{t> \s_{k}, \o_{x^{(k-1)}}(t)=0\},$
we get
\begin{align*}
\label{eq:cutoff7}
\max_\o\bbP_\o\big(\t_x\ge \hat T_n\big)
&\le \max_\o\bbP_\o\big(\exists k: \s_{k-1}-\s_{k} \ge \varphi^{(k-1)}_n(x)\big)\\
&\le \sum_{k}\max_{\o:\, \o_{x^{(k)}}=0}\bbP_\o\big(\s_{k-1}-\s_{k}\ge \varphi^{(k-1)}_n(x)\big).
\end{align*}
As in the $d=2$ case, we apply Lemma \ref{lemma:LD} to each term in the above sum with $k<d-1$ and \cite[Theorem 2]{Ganguly-Lubetzky-Martinelli}) to the term $k=d-1$ to conclude that the r.h.s. above is smaller than $e^{-c(q) \log(n)^2}$.  \qed

\section{Proof of Proposition \ref{prop:1}}
\label{sec:Proof of Proposition 3.1}
\subsection{Proof of (i)}
We proceed in two steps: we first prove that $\phi(0;d)\ge 1/d$ using
a bottleneck argument and
then, inspired by \cite{CFM2}, that $\phi(0;d)\le 1/d$.
\\

\paragraph{{\bf The lower bound.}}
Let $\L$ be the equilateral box of side length $\lfloor
2^{\theta_q/d}\rfloor$ and let $\L\supset V\supset \{0,x_\L\}$ be such that $\g(V)>0$.
\begin{claim} For any $\e>0$ there exists $q(\e)>0$ such that for any $q\le q(\e)$
$$
\g(V)\le 2^{-(1-\e)\frac{\theta_q^2}{2d}}.
$$
\end{claim}
\begin{proof}
Let $A_*\subset
\O_\L$ be the event defined in \cite[Definition 4.3]{CFM2} and let
$A_V=\{\o\in \O_V:\ 1_{V^c}\cdot \o\in A_*\},$ where
$1_{V^c}$ denotes the configuration in $\O_{V^c}$ identically equal
to one. As observed in \cite[Remark 4.4]{CFM2} $1_V\notin
A_V$ while the configuration with exactly one vacancy at $x_\L$
belongs to $A_V$. Therefore,
$\var\big(\id_{A_V}\big)\ge (1-q)^{2|V|-1}q=\Theta(q)$ because $|V|\le
1/q$. Next we bound the Dirichlet form of $\id_{A_V}$. Let $\partial
A_V$ consists of those elements of $A_V$ which are connected to
$A_V^c$ via a legal update for the East chain on $V$. Then 
\[
  \cD_V(\id_{A_V})
  \le |\L|\mu_V(\partial A_V)\le |\L|\mu_{V^c}(1_{V^c})^{-1}\mu_\L(\partial A_*)\le 2^{-(1-o(1))\frac{\theta_q^2}{2d}},
\]
where we used \cite[Section 4.3]{CFM2}. The claim now follows from 
the variational characterization of the spectral gap $\g(V)$.
\end{proof}
Since the box $\L$ is $(0,1;\theta_q)$-outstretched, the claim implies
that if $\l\sim \cH(0)$ then $\l\ge 1/d$. Hence $\phi(0;d)\ge 1/d$. 
\\

\paragraph{\bf The upper bound.} 
The proof that $\phi(0;d)\le 1/d$ requires a bootstrap 
procedure like the one introduced in \cite{CFM2}. The
\emph{base case} is Lemma \ref{lem:East} which gives that $\l=1\sim \cH(0)$. We
then prove the \emph{recursive step}, namely that $\l\sim \cH(0)$ implies $F(\l)\sim \cH(0)$, where
\begin{equation}
  \label{eq:mapping}
  F(\l)=((2d-1)\lambda-1)/(d^2\lambda-1)<\l\quad \forall \l\in [1/d,1].
\end{equation}
Since the mapping $F$ has an attractive fixed point in $1/d,$ the sought claim
follows by iteration. 

\paragraph{\it Proof of the recursive step} We find it easier to work with
\emph{equilateral} boxes, i.e.\ $(0,1;\theta_q)$-outstretched
boxes. For this purpose we first introduce a new condition, equivalent
to $\cH(0),$ which only requires a check on the spectral gap of suitable subsets
of equilateral boxes. 

\begin{definition}
        We say that $\l\sim \cH'(0)$ if
        $\forall\,\e>0$ there exists $q(\e)>0$ such that 
        $\forall\,q\le q(\e)$ and for any equilateral box
        $\L$ there exists $\L\supset V\supset \{0,x_{\Lambda}\}$ such that $\gamma(V)\ge
                        2^{-\l (1+\e)\frac{\theta_q^2}{2}}$.
\end{definition}
\begin{lemma}\label{lemma:hprime_h0}
        $\lambda\sim \cH'(0)$ iff $\lambda\sim \cH(0)$. 
      \end{lemma}
\noindent The proof of the lemma is postponed to the appendix.
Next, motivated by \cite[Definition~5.2]{CFM2}, we construct  three
useful auxiliary Markov
chains. The first one, dubbed the \emph{*East chain}, is a natural generalisation of the East chain when the single site state space is a general finite set and not just the set $\{0,1\}$. The other two chains, dubbed the \emph{Knight Chain} and \emph{*Knight Chain} respectively, require a somewhat more involved geometric setting.

\begin{definition}[The *East chain]
\label{def:*East chain} Let $q^*\in (0,1)$ and let ${\{\Omega^*_x,\mu_x^*\}}_{x\in \bbZ^d_+}$ be a family of finite probability spaces. For each $x\in \bbZ^d_+$ let $G^*_x\subset \Omega^*_x$ be an event such that
$\mu^*_x(G^*_x)=q^*.$ In the sequel we will refer to $G_x^*$ as the \emph{facilitating event} at $x$. Let $V\subset \Z_+^d$ be a finite subset that
contains the origin. Then the \emph{$^*$East chain}
on $\Omega^*_V:=\otimes_{x\in V}
\Omega^*_x$
is the continuous time Markov chain, reversible w.r.t.\
$\mu_V^*=\otimes_{x\in V}\mu^*_x,$ evolving as follows. With rate one and independently across $V$ the chain attempts to update its current state $\o_x$ at any given vertex $x\in V$ by proposing a new state $\o^{\rm new}_x$ sampled from $\mu_x^*$. The attempt is successful, i.e. the proposal is accepted iff the constraint $c_x^*(\omega)=1$ where 
\begin{align*}
        c_x^*(\omega)=
        \begin{cases}
                1 &\text{if } x=0\ \text{or}\ \exists \, \mathbf e\in \mathcal{B} \ \text{such that}\
                x-e\in V\ \text{and}\ \omega_{x-\mathbf e}\in
                G^*_{x-\mathbf e},\\
                0 & \text{else.}
        \end{cases}
\end{align*}
\end{definition}
\begin{remark}
If for all $x$ the probability space $\{\O_x^*,\mu_x^*\}$ and the facilitating event $G_x$ coincide with the two points space $\{\{0,1\},\text{Bernoulli(p)}\}$ and with the event $\o_x=0$ respectively, then the *East chain coincides with the standard East chain discussed so far. However, as we will see in the proof of Proposition \ref{lemma:iterate_bound}, in a natural renormalisation procedure in which $\bbZ^d_+$ is partitioned into equal disjoint blocks indexed by $x\in \bbZ^d_+$ and the $0/1$ variables associated to the vertices of each "block" are treated together as a single block-variable, the natural choice for the pair $\big(\O_x^*,\mu_x^*\big)$ is the probability state space $\big(\{0,1\}^{B_x}, \otimes_{i\in B_x}\mu_i\big)$. In this case the natural candidate for the facilitating event $G_x$ is the event that inside the block $B_x$ there is at least one vacancy.     
\end{remark}
As in \cite[Proposition~3.4]{CFM2} it is possible to prove that the 
spectral gap $\gamma^{*}(V)$ of the $^*$East chain in $V$ coincides with the spectral gap
$\gamma(V;q^*)$ of the \emph{standard} East chain with vacancy density
$q^*$.

The construction of the Knight chain and *Knight chain requires first the construction of the
    Knight graph (see \cref{fig:K1}).
\begin{figure}[ht!]
        \begin{center}
                \begin{tikzpicture}[scale=0.95]
                        \begin{scope}[shift={(-4,1.5)}]
                        \draw[gray,step=0.5] (-0.25,-0.25) grid
                                (2.75,2.75);
                                 \draw[fill=gray,opacity=0.5,gray] (0,1.5) -- (1,1.5) --
                                         (0,2.5) -- cycle;
                                \draw[opacity=0.75,gray] (0.5,1.5) node {$\bullet$};
                                \draw[opacity=0.75,gray] (1,1.5) node {$\bullet$};
                                \draw[opacity=0.75,gray] (0,2) node {$\bullet$};
                                \draw[opacity=0.75,gray] (0.5,2) node {$\bullet$};
                                \draw[opacity=0.75,gray] (0,2.5) node {$\bullet$};
                                \draw[thick] (-0,2.25) node[anchor=east] {$E_x$};
                                \draw[thick] (-0.2,1.5)
                                node[anchor=east] {$x$};
                                \draw[thick] (0,-0.2) node[anchor=east] {$0$};
                                \draw[thick, cross] (0,0) node {$\bullet$};
                                \draw[thick, cross] (0,1.5) node {$\bullet$};
                                \draw[thick, cross] (0.5,1) node {$\bullet$};
                                \draw[thick, cross] (0.5,2.5) node {$\bullet$};
                                \draw[thick, cross] (1,0.5) node {$\bullet$};
                                \draw[thick, cross] (1,2) node {$\bullet$};
                                \draw[thick, cross] (1.5,0) node {$\bullet$};
                                \draw[thick, cross] (1.5,1.5) node {$\bullet$};
                                \draw[thick, cross] (2,1) node {$\bullet$};
                                \draw[thick, cross] (2,2.5) node {$\bullet$};
                                \draw[thick, cross] (2.5,0.5) node {$\bullet$};
                                \draw[thick, cross] (2.5,2) node {$\bullet$};
                                \draw[thick,cross] (1.25,-1) node {(A)};
                                \draw[shorten <= 0.1cm, shorten >=0.1cm] (1.5,1.5) to (2.5,2);
                                \draw[shorten <= 0.1cm, shorten >=0.1cm] (1.5,1.5) to (2,2.5);
                                \draw[shorten <= 0.1cm, shorten >=0.1cm] (1.5,1.5) to (0.5,1);
                                \draw[shorten <= 0.1cm, shorten >=0.1cm] (2.5,2) to (2,1);
                                \draw[shorten <= 0.1cm, shorten >=0.1cm] (2,1) to (1,0.5);
                                \draw[shorten <= 0.1cm, shorten >=0.1cm] (2,1) to (1.5,0);
                                \draw[shorten <= 0.1cm, shorten >=0.1cm] (1.5,1.5) to (1,0.5);
                                \draw[shorten <= 0.1cm, shorten >=0.1cm] (1.5,0) to (2.5,0.5);
                                \draw[shorten <= 0.1cm, shorten
                                >=0.1cm] (0,0) to (1,0.5);
                                \draw[shorten <= 0.1cm, shorten >=0.1cm] (0,0) to (0.5,1);
                        \end{scope}
                        \begin{scope}[scale=0.8]
                        \draw[gray,step=0.5cm] (0,0) grid (6.5,6.5);
                        \foreach \i in {0,0.5,...,2}{
                                \foreach \j in {0,0.5,...,2}{ 
                                        \draw (2*\i+\j,\i+2*\j) node {$\bullet$};
                                }
                        }
\draw (0,0)--(2,4)--(6,6)--(4,2)--(0,0);
\foreach \i in {0,1,...,4}{
  \draw [thin] (\i,0.5*\i)--(\i+2,0.5*\i+4);
  \draw [thin] (0.5*\i,\i)--(0.5*\i+4,\i+2);
  }
                        \draw[thick,cross] (1.25,3) node[anchor=east] {$\Lambda^K$};
                        \draw[thick,cross] (6.5,5.8) node[anchor=north] {$x_{\Lambda^K}$};
                        \draw[thick, cross] (3.25,-0.5) node {(B)};

                        \draw[->,thick] (3.6,2.6) to [bend left=45] (9.4+0.5,2.1);
                        \draw[thick, cross] (7,4) node[anchor=south] {$\Phi$};
                        \begin{scope}[shift={(8.5,1.5)}]
                                \draw[thick, cross] (1,-0.5) node {(C)};
                                \draw[gray,step=0.5] (0,0) grid (2,2); 
                                \draw[thick,cross] (2,2)
                                        node[anchor=south west] {$\Phi(x_{\Lambda^K})$};
                                \draw[thick,cross] (2.2,1) node[anchor=west] {$\Phi(\Lambda^K)$};
                        \end{scope}
                        \end{scope}
                \end{tikzpicture} 
        \end{center}
                \caption{(A) A piece of the Knight graph (the black
                  dots and the Knight edges) for $d=2$.
                The gray triangle corresponds to
                the enlargement $E_x$ of $x$. (B) The graph of the largest Knight equilateral box
                $\Lambda^K$ of side length $4$ inside an equilateral
                box of side length $13$. (C) Under the natural
                isomorphism $\Phi$ the graph $\Lambda^K$ becomes an equilateral
                box.}
              \label{fig:K1}
            \end{figure}
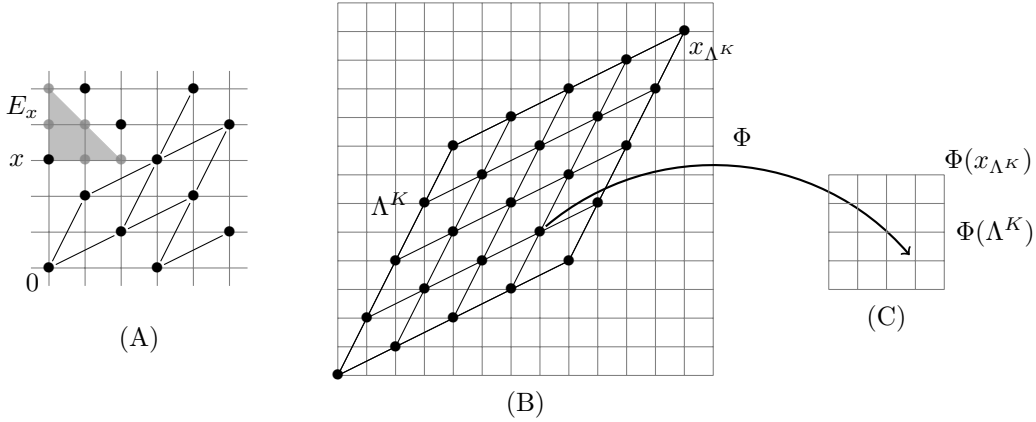
\begin{definition}[The Knight graph]
Given two vertices $x,y\in \bbZ^d$ we say that they form a Knight edge
        if there exists $j\in [d]$ such that $y_i=x_i-1$ for all $i\neq j$ and
        $y_j=x_j-2$ or vice versa. We then consider the unique graph $G=(W,E),
        W\subset \bbZ^d,$ constructed as follows. The vertex set $W$ contains
        the origin and those $x\in \bbZ^d$ which are connected to the origin
        via a path of Knight edges. The edge set $E$ consists of all the Knight
        edges of $W\times W$.  It is easy to see that $G$ is isomorphic to
        $\bbZ^d$ via the natural isomorphism $\Phi$ which is unique if we set
        $\Phi(0)=0$.
\end{definition}
The graph $G$ will inherit  the notation used so far for $\bbZ^d$ via the
       isomorphism $\Phi$. We write $W_+=\Phi^{-1}(\bbZ^d_+)$ and we say that $\L^K\subset W_+$
       is a Knight equilateral box containing the origin if
       $\Phi(\L^K)$ is an equilateral box in $\bbZ^d_+$ containing the origin. In the
       latter case we write $x_{\L^K}\in \L^K$ for the vertex $\Phi^{-1}(x_{\Phi(\L^K)})$. Notice that $\exists\, c>0$ such that for any equilateral box $\Lambda\subset \bbZ^d_+$
containing the origin there exists a Knight
equilateral box $\L\supset \Lambda^{K}\ni 0$ such that $\|x_{\Lambda}-x_{\Lambda^{K}}\|_1\le c.$

Recall that
$\|z-z'\|_1=d+1 \ \forall z,z'\in W$ connected by a Knight edge and $\forall x\in W$ let
$E_x=\{y\in W^c: y\succ x,\ \|x-y\|_1\le d\}$ be the enlargement of
$x$ (see Figure \ref{fig:K1}). The enlargement of a subset $V^K$ of the Knight graph $W$ is the set $EV^K=\cup_{x\in
  V^K}E_x$.
  
We are now ready to define the Knight and *Knight chains. As in Definition \ref{def:*East chain} we assume that we are given $q^*\in (0,1),$ a family ${\{\Omega^*_x,\mu_x^*\}}_{x\in \bbZ^d_+}$ of finite probability spaces and a facilitating event  $G^*_x\subset \Omega^*_x$ for each $x\in \bbZ^d_+$. 
\begin{definition}[The Knight chain ] 
Given an equilateral box $\L\subset \bbZ^d_+$ with origin at $0$ and $V\subset
       \L$ containing the origin, let 
       $V^K:=\Phi^{-1}(V)$. Then the Knight chain on $\O^*_{V^K}$ 
       is the
       image under $\Phi^{-1}$ of the $^*$East chain on $\O^*_V$.
\end{definition}
\begin{definition}[ The *Knight chain]
 \label{def:*K} 
            Given an equilateral box $\L\subset \bbZ^d_+$ with origin at $0$ and $V\subset
       \L$ containing the origin the *Knight chain on $\O^*_{EV^K\cap \L}$ is the continuous
       time Markov chain evolving as
       follows. At any legal update at $z\in V^K$ of the Knight chain
       on $\O^*_{V^K}$  
       the whole configuration in $E_z\cap \L$ is resampled from
       $\mu^*_{E_z\cap \L}$.
\end{definition} 
It is immediate to verify that the *Knight chain is reversible
w.r.t.\ $\mu^*_{EV^K\cap \L}$ with a positive spectral gap
$\g^{*K}(EV^K\cap \L).$ In the
appendix will prove the following result:
\begin{lemma}
\label{lemma:K}  
$
\g^{*K}(EV^K\cap \L)=\g(V;q^*).  
$
\end{lemma}
We can finally state the
main result of this section.
\begin{proposition}\label{lemma:iterate_bound}
Fix $\lambda\in (1/d,1]$ and let $F(\cdot)$ be the mapping in
\eqref{eq:mapping}. Then $\l\sim \cH'(0)$ implies that $F(\l)\sim\cH'(0)$.
\end{proposition}
\begin{proof}
Let $\lambda\in (1/d,1]$ with  $\l\sim \cH'(0)$ and let $\Lambda\subset \bbZ^d_+$ be an equilateral box with side
        length $L$. Using a suitable $\l$-dependent *Knight chain, we will now construct a set $V\subset \L$ such that
        $\g(V)\ge 2^{-F(\lambda)\frac{\theta_{q_*}^2}{2}(1+\e)}$.

        Let $\ell=\lfloor 2^{m\theta_q}\rfloor,$ where
        $m=(d\lambda-1)/(d^2\lambda-1)$ and observe that $\ell\le
        2^{\theta_q/d}.$ If $L\le \ell$
        we can use Lemma \ref{lem:East} to get that
        \[
          \g(\L)\ge 2^{-(m-m^2/2)\theta_q^2(1+o(1))}\ge 2^{-F(\l)\frac{\theta_q^2}{2}(1+o(1))}.
        \]
In this case we simply choose $V=\L$. If instead $L > \ell$ we
proceed as follows. 
\begin{figure}[ht!]
        \begin{center}
        \begin{tikzpicture}[scale=0.25]
                \draw[lightgray,opacity=0.5] (0,0) grid (30,30);
                \draw[fill=lightgray, opacity=0.1]  (0,0) -- ++(27,0) -- ++(0,27) -- ++(-27,0) --
                        cycle;
                \foreach \i in {0,...,2}{
                        \foreach \j in {0,...,2}{ 
                                \draw[opacity=0.4,fill=red,red] (8*\i+4*\j,4*\i+8*\j) -- ++(3,0) -- ++(0,3)
                                        -- ++(-3,0) -- cycle;
                        }
                }
                \foreach \i in {0,...,2}{
                        \foreach \j in {0,1}{ 
                                \draw[thick,fill=green, opacity=0.4]
                                (8*\i+4*\j,4*\i+8*\j) -- ++(3,0) --
                                ++(0,3) -- ++(-3,0) -- cycle;
                                \foreach \offset in {1}{
                                        \draw[pattern=north west lines]
                                                (8*\i+4*\j+4*\offset,4*\i+8*\j) -- ++(3,0) -- ++(0,3)
                                                -- ++(-3,0) -- cycle;
                                        \draw[pattern=north west lines]
                                                (8*\i+4*\j,4*\i+8*\j+4*\offset) -- ++(3,0) -- ++(0,3)
                                                -- ++(-3,0) -- cycle;
                                }
                                \draw[pattern=north west lines] (8*\i+4*\j+4,4*\i+8*\j+4) -- ++(3,0) -- ++(0,3) -- ++(-3,0) -- cycle;
                        }
                }

                \foreach \i in {0,...,2}{
                        \foreach \j in {0}{ 
                                \foreach \offset in {2}{
                                        \draw[pattern=north west lines]
                                                (8*\i+4*\j+4*\offset,4*\i+8*\j) -- ++(3,0) -- ++(0,3)
                                                -- ++(-3,0) -- cycle;
                                        \draw[pattern=north west lines]
                                                (8*\i+4*\j,4*\i+8*\j+4*\offset) -- ++(3,0) -- ++(0,3)
                                                -- ++(-3,0) -- cycle;
                                }

                        }
                }
                \foreach \i in {0,...,1}{
                        \foreach \j in {1}{ 
                                \foreach \offset in {2}{
                   \draw[pattern=north west lines]  (8*\i+4*\j+4*\offset,4*\i+8*\j) -- ++(3,0) -- ++(0,3)
                                                -- ++(-3,0) -- cycle;
                                        \draw[pattern=north west lines]
                                                (8*\i+4*\j,4*\i+8*\j+4*\offset) -- ++(3,0) -- ++(0,3)
                                                -- ++(-3,0) -- cycle;
                                }

                        }
                }
                \draw[pattern=north west lines]
                        (8+4+8,4+8+12) -- ++(3,0) -- ++(0,3) -- ++(-3,0) -- cycle;

                \draw[thick,fill=green, opacity=0.4] (8*2+4*2,4*2+8*2) -- ++(3,0) -- ++(0,3)
                        -- ++(-3,0) -- cycle;

                \foreach \i in {0,...,6}{
                        \foreach \j in {0,...,6}{
                                \draw (4*\i,4*\j) -- ++(3,0) -- ++(0,3)
                                        -- ++(-3,0) -- cycle;
                        }
                }
                                       \draw[cross] (30,30) node
                                       {$\bullet$};
                                       \draw[cross] (27-3,27-3) node {$\bullet$};
                                       \draw[cross] (24,25) node {$\bullet$};
                \draw (30,27) node[anchor=west] {$\G$};
                \draw (31,15) node[anchor=west] {$\Lambda$};
                \draw (21.5,1.5) node {$B_{\mathbf{j}}$};
                \draw (30,30) node[anchor=south] {$x_{\Lambda}$};
                \draw (37.5,24 ) node {$(\lfloor L/\ell\rfloor,\dots, \lfloor L/\ell\rfloor)$};
                 \draw[->, bend left=10] (24,24) to  
                 (31,24);
                 \draw [thick] (24,25)--(30,25)--(30,30);
               \end{tikzpicture}
        \end{center}
        \caption{The setting in the proof of Proposition
          \ref{lemma:iterate_bound} with $\ell=3$ and $L=30$. The $3\times 3$
        boxes $B_{\mathbf j}$ are those with $\mathbf j\in \L_B,$ the
        coloured (red/green) ones are those with $\mathbf j\in
        \L_B^K, $ the green ones are those with $\mathbf j\in
        V^K, $ and the dashed ones are those with $\mathbf j\in
        (EV^K\cap \L_B)\setminus V^K$. The set $V$ with $\g(V)\ge
        2^{-F(\l)\frac{\theta_q^2}{2}(1+o(1))}$ is the union of
        the green and dashed boxes together with the path $\G.$ }
      \label{fig:K2}
    \end{figure}

Let $B_0$ be the equilateral box with side length
        $\ell,$ let $\Lambda_B:= \{0,\dots, \lfloor L/\ell\rfloor\}^d$ and for
        $\mathbf{j}\in \bbZ_+^d$ let $B_{\mathbf{j}}= B_0 + \mathbf{j}\ell$.
       Thus $\cup_{\mathbf{j}\in \Lambda_B}B_{\mathbf{j}}\subset
        \Lambda$ and $\min_{x\in
        B_{\mathbf{j}_{\Lambda_B}}}\|x-x_{\Lambda}\|_1\le O(\ell)$. We
      say that
        $B_{\mathbf{j}}$ is \emph{good} if it contains at least one
        vacancy and observe that the density $q^*=1-{(1-q)}^{\ell^d}$ of good
        boxes satisfies (we use the Bonferroni inequality for the
        lower bound)
        \[
          q\ell^d/2 \le q^*\le q\ell^d\le 1 \quad \Rightarrow \quad
          \theta_{q^*}\in [ \theta_{q}(1-dm), \theta_{q}(1-dm)+1].
          \]
In the sequel we will use the Knight chain and the *Knight
chain with $\O^*_{\mathbf j}=\{0,1\}^{B_{\mathbf{j}}},\ \mu^*_{\mathbf j}=\otimes_{x\in B_{\mathbf j}}
\mu_x$, and facilitating events $G^*_{\mathbf j}=\{\text{$B_{\mathbf j}$
  is good}\}$. 

Let $\L_B^K\subset \L_B$ be the largest Knight equilateral box containing the
origin and for $V^K\subset \L_B^K$ consider the *Knight
chain on $\O_{EV^K\cap \L_B}$.
Using $\l\sim \cH'(0)$ we can choose $V^K\subset \L^K_B$ such that $V^K\supset
\{0,{\mathbf j}_{\L^K_B}\}$ and $\forall \e>0$ and $q$ small enough depending on
$\e$ 
\begin{equation}
  \label{eq:26}
  \gamma^{*K}(EV^K\cap \L_B)=\g(\Phi(V^K);q^*)
                \ge
                2^{-\l\frac{\theta_{q_*}^2}{2}(1+\epsilon/2)},
\end{equation}
where in the equality we used Lemma \ref{lemma:K}.
We then take $V=V_1\cup \G \subset \L,$ where $V_1= \cup_{\mathbf j\in EV^K\cap
  \L_B}B_{\mathbf j}$ and 
$\G=(x^{(0)},x^{(1)},\dots, x^{(N)})$ is any path in
$\L$ satisfying: (i) $x^{(0)}\in V_1,\, x^{(N)}=x_\L,$ (ii) $x^{(i-1)}\prec x^{(i)} \
\forall i\in [N],$ and (iii) $N=O(\ell)$.
By construction such a path always exists. 
\begin{claim}
  For any $\e>0$ there exists $q(\e)>0$ such that for all $q\le q(\e)$
  \[
\g(V)\ge 2^{-\frac 12 (\l \theta_{q_*}^2+(2m-m^2)\theta_q^2)(1+\epsilon)}=2^{-F(\l) \frac{\theta_q^2}{2}(1+\epsilon)}.
  \]
  
\end{claim}
Clearly the claim proves the proposition. 
\end{proof}
\begin{proof}[Proof of the claim]
  Fix $\e>0$ and choose $q$ small enough depending on $\e$.
Let $V_2=\G\setminus x^{(0)}$ and use Lemma \ref{lem:2 blocks} to get that
 $\g(V)\ge 2^{-(\theta_q+2)}\min\big(\g(V_1),\g^\s(V_2)\big)$ where $\s\in
 \O_{\partial_{\downarrow}V_2}$ consists of a unique vacancy at
$x^{(0)}.$ 
Lemma \ref{lem:trick} together with \eqref{eq:26}
and the fact that $\g(B_0)\ge 2^{-(2m-m^2)\frac{\theta_q^2}{2}(1+\e)}$ give
that $ \g(V_1)\ge 2^{-\frac 12 (\l
  \theta_{q_*}^2+(2m-m^2)\theta_q^2)(1+\epsilon)}$.
Moreover, using Lemma \ref{lem:East} we have that $\g^{\s}(V_2)\ge
2^{-(2m-m^2)\frac{\theta_{q}^2}{2}(1+\e)}$. The claim then follows
from the observation that
\[
  (\l  \theta_{q_*}^2+(2m-m^2)\theta_q^2)=F(\l)\theta_q^2(1+o(1))\quad
  \text{as $q\to 0$.}
\]
\end{proof}
The recursive step $\l\sim \cH(0)\Rightarrow F(\l)\sim
\cH(0)$ now follows immediately from Lemma \ref{lemma:iterate_bound} and Lemma \ref{lemma:hprime_h0}.
\subsection{Proof of (ii)}
The proof consists of two different steps. We first prove
that  $\phi(\b;2)<1$ for all $\b<1$ implies that the same
holds for any $d\ge 3$ and then we deal with the
two dimensional case.
\subsubsection{The induction step}
Fix $d\ge 3$ and $\b<1$ and assume $\phi(\b;d')<1$ for any $2\le d'\le
d-1$. We are going to prove that $\phi(\b;d)<1$ as well. Fix  $\k\ge
1$ together
with a $(\b,\k)$-outstretched box
$\L$ with side lengths $(L_1,\dots,L_d)$ and 
set (see \cref{fig:2})
   \begin{align*}
\L_1&=\{x\in \L:\ x_1\le \lfloor
      L_1/2\rfloor, \ x_d=0\},\\
  \L_2 &=\{x\in \L:\ x_1> \lfloor
L_1/2\rfloor, x_i=L_i,\  2\le i\le d-1\}.
  \end{align*}
\begin{figure}[!ht]
  \begin{center}
\begin{tikzpicture}{scale=0.7}
        \newcommand{\Side}{2}
        \coordinate (O) at (0,0,0);
        \coordinate (A) at (0,\Side,0);
        \coordinate (B) at (0,\Side,\Side);
        \coordinate (C) at (0,0,\Side);
        \coordinate (D) at (\Side,0,0);
        \coordinate (E) at (\Side,\Side,0);
        \coordinate (F) at (\Side,\Side,\Side);
        \coordinate (G) at (\Side,0,\Side);

        \draw[black,fill=gray!50] (O) -- (\Side/2-0.1,0,0) -- (\Side/2-0.1,0,\Side) -- (C) -- cycle;
        \draw[black,fill=gray!50] (\Side/2,0,0) -- ++(0,\Side,0) -- ++(\Side/2,0,0)
                -- (D) -- cycle;

        \draw[black] (O) -- (C) -- (G) -- (D) -- cycle;
        \draw[black] (O) -- (A) -- (E) -- (D) -- cycle;
        \draw[black] (O) -- (A) -- (B) -- (C) -- cycle;
        \draw[black] (D) -- (E) -- (F) -- (G) -- cycle;
        \draw[black] (C) -- (B) -- (F) -- (G) -- cycle;
        \draw[black] (A) -- (B) -- (F) -- (E) -- cycle;

        \node[anchor=north] at (C) {$0$};
        \filldraw (\Side/2,0,0) circle (1pt);
        \filldraw (\Side/2-0.1,0,0) circle (1pt);
        \node[anchor=south] at (E) {$x_{\Lambda}$};
        \draw[thick, cross] (-0.1,-0.5) node[anchor=south west] {$\L_1$};
         \draw[thick, cross] (1.3,1) node[anchor=south west] {$\L_2$};
\end{tikzpicture}
\end{center}
\caption{The boxes $\L_1,\L_2$. The two black dots denote
  $x_{\L_1}$ and the origin of $\L_2$ at $x_{\L_1}+\mathbf e_1$ respectively.}
\label{fig:2}
\end{figure}
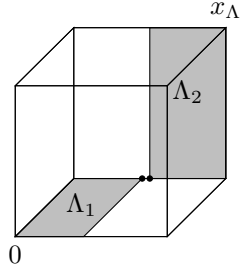
 By construction, the origin of the box
$\L_2$ is at $x_{\L_1}+
\mathbf e_1$ and $x_{\L_2}=x_\L.$ Moreover,
both $\L_1$ and $\L_2$ are $(\b,\k)$-outstretched boxes in $\bbZ_+^{d-1}$
and $\bbZ_+^2$ respectively.
The induction hypothesis implies that for all $\e>0$ and all $q$ small enough depending on
$\e,\b,\k$ there exist $V_i\subset \L_i, i=1,2,$ such that
\begin{itemize}
\item $V_1\supset \{0,x_{\L_1}\}$ and $V_2\supset \{x_{\L_1}+\mathbf
  e_1, x_\L\};$
  \item $\g(V_1)\ge 2^{-(1+\e)\phi(\b;d-1)\frac{\theta_q^2}{2}}$ and
    $\g^\s(V_2)\ge 2^{-(1+\e)\phi(\b;2)\frac{\theta_q^2}{2}},$ where
    $\s\in \O_{\partial_{\downarrow}V_2}$ has a unique vacancy at $x_{\L_1}$.
  \end{itemize}
Lemma \ref{lem:2 blocks} then implies that
$ \g(V) \ge 2^{-(1+2\e)(\phi(\b;d-1)\vee
        \phi(\b;2))\frac{\theta_q^2}{2}},$
 i.e.\ $\phi(\b;d)\le \phi(\b;d-1)\vee
        \phi(\b;2))<1$.    
\subsubsection{The base case \texorpdfstring{$d=2$}{d=2}}
  We will
  prove that $\forall \b\in (0,1)$ 
  \begin{equation}
    \label{eq:20}
    \phi(\b;2)\le \frac 12 (1-\b)^2 +2\b -\b^2,
  \end{equation}
  which, in particular, implies that $\phi(\b;2)<1\ \forall \b<1$. 
The main idea here is to partition a $(\b,\k)$-outstretched box $\L$ into
suitably chosen mesoscopic boxes in such a way that the coarse-grained
version of $\L$ becomes a $(0,2)$-outstretched box on which the control of
the Dirichlet eigenvalue gap is assured by part (i) of the
proposition.

Fix $0<\b<1, \k\ge 1$ together with
a $(\b,\k)$-outstretched box $\L$ with side lengths $(L_1,L_2),$ and
assume w.l.o.g. that $L_1=\min_i L_i$.
We set $\ell=\lceil 
(L_2+1)/2(L_1+1)\rceil\le (\k/2)2^{\b \theta_q},$ and
w.l.o.g.\  we assume that $(L_2+1)/\ell \in \bbN$.
We then partition
$\L$ into vertical one dimensional boxes $B_{\mathbf j}=B+ x_{\mathbf
  j}, \ B=\{0\}\times
\{0,\dots,\ell-1\}, x_{\mathbf j}=(j_1,j_2\ell)$ where $\mathbf j\in Q=\{0,\dots, L_1-1\}\times \{0,\dots, (L_2+1)/\ell -1\}$. 
 We also write $\O^*_{\mathbf
  j},\mu^*_{\mathbf j}$ for $\O_{B_{\mathbf
    j}}$ and $\mu_{B_{\mathbf j}}$ respectively.

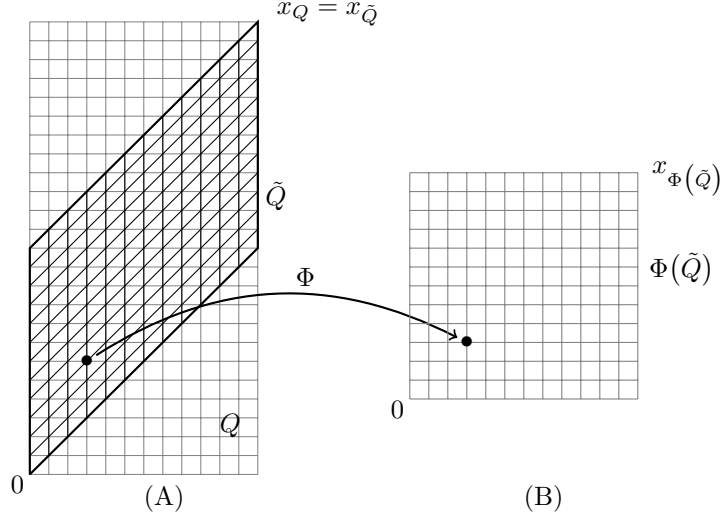
\begin{figure}[!ht]
  \begin{center}
\begin{tikzpicture}[scale=0.5]
        \draw[step=0.5cm,gray,very thin] (0,0) grid (6,12);
        \draw[thick] (0,0) --(6,6) --(6,12) --(0,6)
        --(0,0);
        \foreach \x in {0,...,12}{
             \draw[black, thin] (0,0.5*\x)--(6,6+0.5*\x);
             \draw[black,thin] (0.5*\x,0.5*\x)--(0.5*\x,0.5*\x+6);
        }
        \draw[thick, cross] (-0.5,-0.5) node[anchor=south west] {$0$};
         \draw[thick, cross] (6.5,12.5) node[anchor=north west]
         {$x_Q=x_{\tilde Q}$};
         \draw[thick, cross] (1.5,3) node {$\bullet$};
        \draw[thick, cross] (5,1) node[anchor=south west] {$Q$};
        \draw[thick, cross] (6.2,7) node[anchor=south west] {$\tilde Q$};
        \draw[thick, cross] (3,-1) node[anchor=south west] {(A)};
        \draw[thick, cross] (13,-1) node[anchor=south west] {(B)};
        \draw[thick,cross] (7,5) node[anchor=south west] {$\Phi$};
\node (a) at (1.5, 3) {};
\node (b) at (11.5, 3.5) {};
\draw[->,thick] (a)  to [bend left=30] (b);
  \begin{scope}[shift={(10,2)}];
    \draw[step=0.5cm,gray] (0,0) grid (6,6);
        \draw[thick, cross] (-0.5,-0.5) node[anchor=south west] {$0$};
         \draw[thick, cross] (8.3,6.3) node[anchor=north east]
         {$x_{\Phi\big(\tilde Q\big)}$};
        \draw[thick, cross] (6.3,3) node[anchor=south west] {$\Phi\big(\tilde Q\big)$};
         \draw[thick, cross] (1.5,1.5) node {$\bullet$};
      \end{scope}
          \end{tikzpicture}

    \end{center}
\caption{\label{fig:3} (A). The box $Q$ and the region $\tilde Q$ with its graph structure. Each
  vertex $\mathbf j\in Q$ represents the box
  $B_{\mathbf j}.$  (B). Under the
  natural isomorphism $\Phi$ the graph $\tilde Q$ becomes the standard
  square graph of $\bbZ^2$.    }

\end{figure}
Let $\tilde Q$ be the subset of $Q$ lying between the two
$45^{\circ}$-lines, one through the origin and the other through the
point $x_Q$ and declare that $\mathbf j, \mathbf j'\in \tilde Q$ form an edge 
if either $j_2=j_2'+1$ and $j_1\in \{j_1',j_1'+1\}$ or vice versa (see
Figure \ref{fig:3}). The
corresponding graph over the vertex set $\tilde Q$ is isomorphic via
the natural graph isomorphism $\Phi$ to
the box $\Phi(\tilde Q)\subset \bbZ_+^2$ with origin
at $x=0$ and side lengths $L_1-1,(L_2+1)/\ell -L_1$. In particular,
we write $\mathbf j'\prec \mathbf j$ iff
$\Phi(\mathbf j')\prec \Phi(\mathbf j).$ 

On any subset $V$ of $\tilde Q$ we
consider the image of the *East chain on $\Phi(V)$ (or rather a slightly
altered version of it as we see below) with parameters
$\O^*_{\mathbf j}, \mu^*_{\mathbf j}$ and facilitating event
$G_{\mathbf j}=\{\o_{B_{\mathbf j}}\neq 1\}.$
Thus $q^*=1-(1-q)^{\ell}$ and $\theta_{q^*}= (1-\b)\theta_q +\Theta(1)$.
As the box $\Phi(\tilde Q)$ is $(0,2)$-outstretched, part (i) of
Proposition \ref{prop:1} implies the existence of $W\subset
\Phi(\tilde Q),$ containing the origin and 
$x_{\Phi(\tilde Q)}$ such that, for any $\e>0$ and any
$q$ sufficiently small depending on $\e,$  
\begin{equation}
  \label{eq:21}
\g(W; q^*)\ge 2^{-(1+\e/2)\frac{\theta_{q^*}^2}{4}}.
\end{equation}
Recall the definition of enlargements $E_x$ from above Definition \ref{def:*K}.
We define $E\Phi^{-1}(W):= \cup_{\mathbf{j}\in \Phi^{-1}(W)}E_{\mathbf{j}}\cap
Q$ and $V=\cup_{\mathbf j\in E\Phi^{-1}(W)}B_{\mathbf j}\subset
\L$ and observe that $V$ contains the origin and the vertex $x_\L$.

\begin{claim}
\label{claim:1}For any $\e>0$ and any
$q$ sufficiently small depending on $\e$
\[
  \g(V)\ge
  \g(W; q^*)\times 2^{-(\b-\b^2/2) \theta_q^2(1+\e)}.
  \]
\end{claim}
\begin{proof}[Proof of the claim]
On $V$ we define an auxiliary dynamics to the *East chain. Consider for that
a partition of $E\Phi^{-1}(W)$ into disjoint connected subsets $U_{\mathbf{j}}$
for $\mathbf{j}\in \Phi^{-1}(W)$ such that $\mathbf{j}\in U_{\mathbf{j}}\subset
E_{\mathbf{j}}$ and $\cup_{\mathbf{j}\in \Phi^{-1}(W)}U_{\mathbf{j}}=
E\Phi^{-1}(W)$. In the sequel we write
$B_{U_{\mathbf{j}}}:=\cup_{\mathbf{j}'\in U_{\mathbf{j}}} B_{\mathbf{j}'}$ and
analogously for $B_{E_{\mathbf{j}}}$. Let
$c_{\mathbf{j}}^{*}(\omega)=1$ iff either $\mathbf{j}=0$ or there exists a
neighbor $\mathbf{j}'\prec \mathbf{j}$ such that there exists at least a
vacancy in $B_{\mathbf{j}'}.$ For such constraints we define the auxiliary dynamics that updates
$B_{U_{\mathbf{j}}}$ with a configuration sampled from
$\mu_{B_{U_{\mathbf{j}}}}$ if $c^*_{\mathbf{j}}(\omega)=1$ and otherwise do
nothing. The spectral gap of this chain is, as the one for the enlarged East
chain, given by $\gamma(W,q^*)$, since the $\mathbf{j}$ that participate in the
dynamics are only the ones in $\Phi^{-1}(W)$ (see the appendix for the proof in
the case of enlarged-*Knight chains). The Poincar\'e inequality reads
\begin{equation}
  \label{eq:22}
  \var_{V}(f)\le \g(W; q^*)^{-1}\sum_{\mathbf j\in
    \Phi^{-1}(W)}\mu_V\big(c^*_{\mathbf j}\var_{B_{U_{\mathbf{j}}}}(f)\big), \quad \forall\, f,
\end{equation}
We now bound a generic term $\mu_V\big(c^*_{\mathbf j}(\o)
\var_{B_{U_{\mathbf j}}}(f) \big)$. Using Lemma \ref{lem:trick}, Lemma \ref{lem:East}, and
$\ell \le O(\kappa)2^{\b \theta_q},$ for any $\e>0$ and any $q$ small enough
depending on $\e$ we get
\begin{align}
  \label{eq:23}
\mu_V\big(c^*_{\mathbf j}(\o)    \var_{B_{U_{\mathbf{j}}}}(f) \big) & \le 2^{(\b-\b^2/2)\theta_q^2(1+\e/2)}\sum_{\substack{z\in
      B_{E_{\mathbf{j}'}}\\ \mathbf{j}'=\mathbf{j}\ \text{or}\ \mathbf{j}'\prec
        \mathbf{j},\, \mathbf j'  \text{ neighbor of } \mathbf j}}\mu_V\big(c^{V}_z\var_z(f)\big).
  \end{align}
 By combining \eqref{eq:22}  and \eqref{eq:23} and using that
 $|E_{\mathbf{j}}|=O(\ell)$ we conclude for $q$ small enough that
\[
    \var_V(f)\le \g(W; q^*)^{-1} \times 2^{(2\b-\b^2)\frac{\theta_q^2}{2}(1+\e)}\cD_V(f)\quad \forall\, f,
\]
and the claim follows from the variational characterization of $\g(V)$.
\end{proof}
The claim together with \eqref{eq:21} finally implies that
$
  \g(V)\ge 2^{-((1-\b)^2/2 +
    2\b-\b^2)\frac{\theta_q^2}{2}(1+\e)},
$
$\forall\, q\le q(\e),$ i.e.\ \eqref{eq:20}.

\subsection{Proof of (iii)}
We already know (cf. Lemma \ref{lem:East}) that $\phi(\b;d)\le 1\
\forall \b.$ Fix now $\b\ge 1$ and consider the $(\b,1)$-outstretched one dimensional box $\L=\cup_{k=0}^{\lfloor
2^{\b \theta_q}\rfloor}\{k\,\mathbf e_1\}$. The only subset $V\subset
\L$
containing the origin and $x_\L$ and such that $\g(V)>0$ is
$V=\L$. But 
$\g(\L)=2^{-\frac{\theta_q^2}{2}(1+o(1))}$ (see again Lemma \ref{lem:East}) so that $\phi(\b;d)\ge 1$.
\appendix
\section{Appendix}
 \label{sec:appendix}
We first state three results which have been used
quite often in the previous sections and then we prove Lemmas \ref{lemma:hprime_h0} and \ref{lemma:K}.
\begin{lemma}
  \label{lem:2 blocks}
Consider two finite sets $V_1,V_2\subset \bbZ^d_+$ such that $V_1\ni
0$ and $\exists\, z\in V_1$ such that $z+\mathbf e\in V_2$ for some
$\mathbf e\in \cB$ and the East chain on $V_2$ with boundary condition
$\s$ having a unique vacancy
at $z$ is ergodic. Then $\g(V_1\cup V_2)\ge \frac q4 \min\big(\g(V_1),\g^\s(V_2)\big)$.    
\end{lemma}
\begin{proof}  Let $V=V_1\cup V_2$ and consider the 
  $2$-block chain on $\O_V$, reversible w.r.t.\ $\mu_V,$:
  \begin{enumerate}[(i)]
  \item with rate one $\o\restriction_{V_1}$ is resampled from
    $\mu_{V_1}$;
    \item with rate one $\o\restriction_{V_2}$ is resampled from
    $\mu_{V_2}$ iff $\o_z=0$.
  \end{enumerate}
  The block chain has Dirichlet form
  \[
    \cD_V^{\rm block}(f)= \mu_V\big(\var_{V_1}(f)+
 \id_{\{\o_z=0\}}\var_{V_2}(f)\big)
 \]
and spectral gap $\g_V^{\rm block}(f)= 1-\sqrt{1-q}\ge q/2$
(see \cite[Proposition 4.4]{CMRT}).
Therefore, the Poincar\'e inequality for the block chain reads
\begin{equation}
  \label{eq:17}
  \var_{V}(f)\le 2/q\,\mu_V\big(\var_{V_1}(f)+
 \id_{\{\o_z=0\}}\var_{V_2}(f)\big) \quad \forall\, f.
\end{equation}
The definition of $\g(V_1)$ and $\g^\s(V_2)$ implies that
\begin{align}
  \label{eq:18}
  \var_{V_1}(f) &\le \g(V_1)^{-1}\sum_{x\in V_1}\mu_{V_1} \big(c_x^{V_1,1}\var_x(f)\big),\\
\label{eq:19} \id_{\{\o_z=0\}}  \var_{V_2}(f) &\le \g^\s(V_2)^{-1}\sum_{x\in
V_2} \id_{\{\o_z=0\}} \mu_{V_2}\big(c_x^{V_2,\o\restriction_{\partial_{\downarrow}V_2}}\var_x(f)\big).
\end{align}
It is now sufficient to insert the r.h.s.\ of \eqref{eq:18}, \eqref{eq:19} into the
r.h.s.\ of \eqref{eq:17} and use
the fact that both $c_x^{V_1,1}, x\in V_1,$ and $\id_{\{\o_z=0\}}
c_x^{V_2,\o\restriction_{\partial_{\downarrow}V_2}}, x\in V_2,$ are
dominated by the constraint $c_x^{V,1}$ 
to conclude that
\[
  \var_{V}(f)\le 4/q\times \max\big( \g(V_1)^{-1},
  \g^\s(V_2)^{-1}\big) \cD_V^{1}(f) \quad \forall\, f,
  \]
where the additional factor of $2$ appears if $V_1\cap V_2\neq \emptyset$.
\end{proof}
\begin{lemma}[{\cite[Lemma 3.1 and eq. (2.9)]{CFM2}}]
  \label{lem:East}
  Consider the box $\L$ with side lengths $(L_1,\dots,L_d)$ and let
  $n\in \bbN$ be such that $\max_i L_i\in (2^{n-1},2^n]$. Then, as
  $q\to 0$
  \[
    \g(\L)=
    \begin{cases}
      2^{-(n\theta_q -\binom{n}{2})(1+o(1))}& \text{ if $n\le
        \theta_q$}\\
      2^{-\frac{\theta_q^2}{2}(1+o(1))} & \text{otherwise.}
    \end{cases}
    \]
\end{lemma}
\begin{lemma}[{\cite[Lemma 3.6]{CFM2}}]
\label{lem:trick}
 Let $\L_x=\L+x$ where $\L$ is a box of $\bbZ^d_+$ and $x$ an
 arbitrary vertex. Let $V\subset
 \L_x$ be such that $x\prec V$ and let
 $A=\{\exists z\in \L_x, z\prec
  V: \o_z=0\}$. Then, 
\[
\mu_{\L_x}\big(\id_A\var_{V}(f)\big)\le \g(\L)^{-1}\cD_{\L_x}(f).
\]
\end{lemma}
\paragraph{\bf Proof of Lemma \ref{lemma:hprime_h0}}
Clearly, $\l\sim \cH(0)\Rightarrow \l\sim \cH'(0)$. 
Suppose now that $\l\sim \cH'(0),$ fix $\kappa\ge 1, \e>0$ and let $\L$ be a
  $(0,\kappa,\theta_q)$-outstretched box with side lengths
$(L_1,\ldots,L_d)$. Let $N=\min_jL_j$ and for any $i\in [d]$ choose a partition
of the discrete interval $\{0,1,\dots,L_i\}$ into $N+1$ discrete intervals,
        $B_0^{(i)},\dots, B_{N}^{(i)}$, ordered from left to right,
        each one
        containing at least one vertex and at most $\k+1$ vertices. For $\mathbf j\in \L_B:=\{0,\dots,N\}^d$ write $B_{\mathbf
        j}=\prod_{i=1}^d B^{(i)}_{j_i}$ so that
      $\cup_{\mathbf j\in \L_b}B_{\mathbf j}=\L.$ Furthermore, let $\Omega^*_{\mathbf{j}}:=\Omega_{B_{\mathbf{j}}},\ \mu^*_{\mathbf{j}}:=\mu_{B_{\mathbf{j}}}$
        and choose as facilitating event $G_{\mathbf
          j}$ the event that the smallest vertex in $B_{\mathbf{j}}$ in the
        $\prec$-ordering (for example the lowest-left corner if $d=2$) has a
        vacancy. Clearly $\mu^*_{\mathbf{j}}(G_{\mathbf j})=q\
        \forall\,\mathbf{j}\in \Lambda_B,$ i.e.\ $q^*=q$. Recall now Definition \ref{def:*East chain}. 
Using $\l\sim \cH'(0)$
      there exists $V^*\subset \L_B$ containing the origin and
      $x_{\L_B}$ such that
      \[
        \g^*(V^*)=\g(V^*;q^*)=\g(V^*)\ge 2^{-\l (1+\e/2)\frac{\theta_q^2}{2}}.
      \]
      Hence, if we set $V=\cup_{\mathbf j\in V^*}B_{\mathbf j}$ and
       write $\var^*$ for the variance w.r.t.\ $\mu^*$ 
      we get
      \begin{align*}
               \var_{V}(f)= \var^*_{V^*}(f)
                \le 2^{\l (1+\e/2)\frac{\theta_q^2}{2}}
                \sum_{\mathbf{j}\in V^*}
                \mu_V(c_{\mathbf{j}}^*\var_{B_{\mathbf j}}(f)).
        \end{align*}
Using Lemma \ref{lem:trick}, \ref{lem:East} and
        the fact that each box $B_{\mathbf j}$ contains at most
        $\kappa^d$ vertices, we get that the r.h.s.\ above is not larger
        than
        $
          2^{\l (1+\e/2)\frac{\theta_q^2}{2}}\,
          2^{O(\kappa^d)\theta_q}\,\cD_\L(f)
         $ so that
         \[
           \var_{V}(f)\le 2^{\l (1+\e/2)\frac{\theta_q^2}{2}}\,
           2^{O(\kappa^d)\theta_q}\cD_V(f)\le 2^{\l (1+\e)\frac{\theta_q^2}{2}}\cD_V(f).
         \]
 Hence, for any $q$ small enough depending on $(\e,\kappa)$,
 $\g(V)\ge  2^{-\l (1+\e)\frac{\theta_q^2}{2}}$ implying that
         $\l\sim \cH(0)$.
\qed
\\

\paragraph{\bf Proof of Lemma \ref{lemma:K}}
Recall Definition \ref{def:*K} and consider a partition ${\{Q_{x}\}}_{x\in V^K}$ of
$(EV^K\setminus V^K )\cap \L$  such that $Q_{x}\subset
E_{x}\ \forall x$. The important point here is that the sets
$\{Q_x\}_{x\in V^K}$ are mutually disjoint, a feature not necessarily shared by
the sets $\{E_x\setminus \{x\}\}_{x\in V^K}$ (see
\cref{fig:K2}). Instead of the *Knight chain on $\O^*_{EV^K\cap \L}$
consider now the (very closely related) chain which at any legal update of
the Knight chain at $x\in V^K$ resamples the whole configuration
in $x\cup Q_x$. This chain can be viewed as a new Knight chain on
$\O^*_{V^K}$ with new parameters $\tilde \O^*_x=\otimes_{z\in x\cup
  Q_x} \O_z^*, \tilde \mu_x^*=\otimes_{z\in x\cup
  Q_x}\mu_z^*, x\in V^K,$ and the same facilitating events as the
original Knight chain. Of course $\otimes_{x\in
  V^K}(\tilde \O^*_x , \tilde \mu^*_x)=
(\O^*_{V^K}, \mu^*_{V^K}).$ Hence, the spectral gap of the new chain, as discussed after
Definition \ref{def:*East chain}, coincides with $\g(V;q^*)$ and $\forall \, f$
\begin{align*}
  \var^*_{V_K}(f)&\le \g(V;q^*)^{-1}\sum_{x\in
    V^K}\mu^*_{V^K}\big(K_x\var^*_{x\cup Q_x}(f)\big)
  \\
  &\le \g(V;q^*)^{-1}\sum_{x\in
    V^K}\mu^*_{V^K}\big(K_x\var^*_{E_x}(f)\big).
\end{align*}
where $K_x$ is the Knight constraint at $x$. Above we used the
fact that   $K_x$ does not depend on $\{\o_z\}_{z\in x\cup Q_x}$
and that $\mu^*_{E_x}\big(\var^*_{x\cup
    Q_x}(f)\big)\le \var^*_{E_x}(f).$ The sum in the
r.h.s.\ above is the Dirichlet form of the *Knight chain and
we conclude that its spectral gap is at least $\g(V;q^*)$. The reverse
inequality follows immediately by projection onto the variables
$\eta_x=1-\id_{\{\o_x\in G^*_x\}}, x\in V^K,$ where $G^*_x$ is the
facilitating event.  
\qed






\end{document}